\documentclass{amsart}
\usepackage[dvips]{epsfig}
\usepackage{graphicx}
\usepackage{latexsym}
\usepackage{amsmath}
\usepackage{amsthm}
\usepackage{amssymb}
\RequirePackage[colorlinks=true,citecolor=blue,urlcolor=blue]{hyperref}

\usepackage{subfigure}

\usepackage{mathtools}

 \usepackage[applemac]{inputenc}

\newtheorem{thm}{Theorem}
\newtheorem{assumption}[thm]{Assumption}
\newtheorem{lem}[thm]{Lemma}

\newtheorem{cor}[thm]{Corollary}
\newtheorem{defi}[thm]{Definition}

\newtheorem{prop}[thm]{Proposition}
\newtheorem{rk}[thm]{Remark}

\newcommand{\toL}{\overset{d}{\longrightarrow}}

\newcommand{\vip}{\vskip.2cm}
\newcommand{\COMMENTAIRE}[1]{}

\newcommand{\field}[1]{\mathbb{#1}}
\newcommand{\EE}{\field{E}}

\newcommand{\GG}{\field{G}}

\newcommand{\PP}{\field{P}}

\newcommand{\RR}{\field{R}}
\newcommand{\TT}{\field{T}}

\newcommand{\VV}{\field{V}}

\newcommand{\Cc}{{\mathcal C}}
\newcommand{\Dd}{\mathcal{D}}

\newcommand{\Ff}{{\mathcal F}}
\newcommand{\Gg}{{\mathcal G}}
\newcommand{\Hh}{{\mathcal H}}

\newcommand{\Nn}{{\mathcal N}}
\newcommand{\Pp}{{\mathcal P}}
\newcommand{\Qq}{{\mathcal Q}}
\newcommand{\Ss}{{\mathcal S}}

\newcommand{\Ttransition}{\boldsymbol{\Pp}}
\def \ep {\varepsilon}

\def \a {\alpha}
\def \b {\beta}

\newcommand{\noisedensity}{G}

\newcommand{\statespace}{{\mathbb R}}

\begin{document}

\title[Estimation in nonlinear bifurcating autoregressive models]{Autoregressive functions estimation in nonlinear bifurcating autoregressive models}

\author{S. Val\`ere Bitseki Penda and Ad\'ela\"ide Olivier}

\address{S. Val\`ere Bitseki Penda, IMB, CNRS-UMR 5584, Universit\'e Bourgogne Franche-Comt\'e, 9 avenue Alain Savary, 21078 Dijon Cedex, France.}

\email{simeon-valere.bitseki-penda@u-bourgogne.fr}

\address{Ad\'ela\"ide Olivier, CEREMADE, CNRS-UMR 7534, Universit\'e Paris-Dauphine, Place du mar\'echal de Lattre de Tassigny, 75775 Paris Cedex 16, France}

\email{adelaide.olivier@ceremade.dauphine.fr}

\begin{abstract}
Bifurcating autoregressive processes, which can be seen as an adaptation of autoregressive processes for a binary tree structure, have been extensively studied during the last decade in a parametric context. In this work we do not specify any {\it a priori} form for the two autoregressive functions and we use nonparametric techniques.
We investigate both nonasymptotic and asymptotic behaviour of the Nadaraya-Watson type estimators of the autoregressive functions.
We build our estimators observing the process on a finite subtree denoted by $\TT_n$, up to the depth $n$. Estimators achieve the classical rate $|\TT_n|^{-\beta/(2\beta+1)}$ in quadratic loss over H\"older classes of smoothness. 
We prove almost sure convergence, asymptotic normality giving the bias expression when choosing the optimal bandwidth. 
Finally, we address the question of asymmetry: we develop an asymptotic test for the equality of the two autoregressive functions which we implement both on simulated and real data.
\end{abstract}

\maketitle





\section{\textsc{Introduction}}

\subsection{A generalisation of the bifurcating autoregressive model}

\subsubsection*{BAR process}

Roughly speaking, bifurcating autoregressive processes, BAR processes for short, are an adaptation of autoregressive processes when the index set have a binary tree structure.
BAR processes were introduced by Cowan and Staudte \cite{CS86} in 1986 in order to study cell division in \textit{Escherichia coli} bacteria.
For $m\geq 0$, let  $\GG_m = \{ 0,1 \}^m $ (with $\GG_0=\{\emptyset\}$) and
introduce the infinite genealogical tree
$$
\TT = \bigcup_{m = 0}^{\infty} \GG_m.
$$
For $u \in \mathbb G_m$, set $|u|=m$ and define the concatenation $u0 = (u,0)\in \mathbb G_{m+1}$ and $u1=(u,1)\in \mathbb G_{m+1}$.
In \cite{CS86}, the original BAR process is defined as follows. A cell $u\in \TT$ with generation time $X_u$ gives rise to the offspring $(u0,u1)$ with generation times
\begin{equation*}
\begin{cases}
X_{u0}  = a + b X_u + \ep_{u0}, \\
X_{u1} = a + b X_u + \ep_{u1},
\end{cases}
\end{equation*}
where $a$ and $b$ are unknown real parameters, with $|b|<1$ which measures heredity in the transmission of the biological feature. The noise sequence $\big((\ep_{u0}, \ep_{u1})\big)_{u \in \TT }$ forms a sequence of independent and identically distributed bivariate centered Gaussian random variables 
and represents environmental effects; the initial value $X_\emptyset$  is drawn according to a Gaussian law.

Since then, several extensions of this model have been studied and various estimators for unknown parameters have been proposed. 
First, one can mention \cite{Guyon} where Guyon introduces asymmetry to take into account the fact that autoregressive parameters for type $0$ or type $1$ cells can differ. Introducing the bifurcating Markov chain theory, Guyon studies the asymptotic behaviour of the least squares estimators of the unknown parameters. He also introduces some asymptotic tests which allow to decide if the model is symmetric or not. 
Several extensions of this linear model have been proposed and studied from a parametric point of view, see for instance Basawa and Huggins \cite{BH99, BH00} and Basawa and Zhou \cite{BZ04, BZ05} where the BAR process is studied for non-Gaussian noise and long memory. 
Around 2010, Bercu, Blandin, Delmas, de Saporta, G\'egout-Petit and Marsalle extended in different directions the study of the BAR process.
Bercu \textit{et al.} \cite{BdSGP09} use martingale approach in order to study least squares estimators of unknown parameters for processes with memory greater than 1 and without normality assumption on the noise sequence.
Even more recently, \cite{dSGPM11,dSGPM12} takes into account missing data and \cite{BB13, Blandin13, dSGPM13} study the model with random coefficients. 
A number of other extensions were also surveyed, one can cite Delmas and Marsalle \cite{DM10} for a generalisation to Galton-Watson trees and Bitseki Penda and Djellout \cite{BD14} for deviations inequalities and moderate deviations. 

\subsubsection*{Nonlinear BAR process} Nonlinear bifurcating autoregressive (NBAR, for short) processes 
generalize BAR processes, avoiding an {\it a priori} linear specification on the two autoregressive functions.  Let us introduce precisely a NBAR process which is specified by {\bf 1)} a filtered probability space $(\Omega,\Ff,(\Ff_m)_{m\geq 0},\PP)$, together with a measurable state space $(\RR,\mathfrak B)$,
{\bf 2)} two measurable functions $f_0,f_1 : \statespace \rightarrow \statespace$ and {\bf 3)} a probability density $G$ on $(\statespace\times\statespace,\mathfrak B\otimes \mathfrak B)$ with a null first order moment. 
%
%
In this setting we have the following

\begin{defi}\label{def:NBAR}
A NBAR process is a family $(X_u)_{u \in \TT}$ of random variables with value in $(\RR,\mathfrak B)$ such that, for every $u\in \TT$, $X_u$ is $\Ff_{|u|}$-measurable and
$$
X_{u0} = f_0(X_u) + \varepsilon_{u0} \quad \text{ and } \quad X_{u1} = f_1(X_u) + \varepsilon_{u1}
$$
where $\big((\varepsilon_{u0}, \varepsilon_{u1})\big)_{u\in \TT}$ is a sequence of independent bivariate random variables with common density~$\noisedensity$.
\end{defi}

The distribution of $(X_u)_{u \in \mathbb T}$ is thus entirely determined by the autoregressive functions $(f_0,f_1)$, the noise density $G$ and an initial distribution for $X_{\emptyset}$. Informally, we view  $\mathbb T$ as a population of individuals, cells or particles whose evolution is governed by the following dynamics: to each $u \in \mathbb T$ we associate a trait $X_u$ (its size, lifetime, growth rate, DNA content and so on) with value in $\mathbb R$. At its time of death, the particle $u$ gives rise to two children $u0$ and $u1$. Conditional on $X_u=x$, the trait $(X_{u0}, X_{u1}) \in \RR^2$ of the offspring of $u$ is a perturbed version of $\big(f_0(x),f_1(x)\big)$. 

The strength of this model lies in the fact that there is no constraint on the form of $f_0$ and $f_1$, whose form is free. This is particularly interesting in view of applications for which no {\it a priori} knowledge is available on the autoregressive links.  \\

When $X_{\emptyset}$ is distributed according to a measure $\mu(dx)$ on $(\RR,\mathfrak B)$, we denote by $\PP_\mu$ the law of the NBAR process $(X_u)_{u \in \TT}$ and by $\EE_\mu[\cdot]$ the expectation with respect to the probability $\PP_\mu$.


\subsection{Nonparametric estimators of the autoregressive functions}

\subsubsection*{Estimation of the autoregressive functions} Our aim is to estimate the unknown autoregressive functions $f_{0}$ and $f_{1}$ in Definition~\ref{def:NBAR} from the observation of a subpopulation. For that purpose, we propose to make use of a Nadaraya-Watson type estimator (introduced independently by Nadaraya \cite{Nadaraya64} and Watson~\cite{Watson64} in 1964). The Nadaraya-Watson estimator is a kernel estimator of the regression function $\EE \big[ Y | X = x \big]$ when observing $n$ pairs $(X_1, Y_1), \ldots,$ $(X_n, Y_n)$ of independent random variables distributed as $(X,Y)$. 
The Nadaraya-Watson estimator was also used in the framework of autoregressive time series in order to reconstruct $\EE \big[ X_{n} | X_{n-1} = x \big]$, assuming that $( X_n)_{n \geq 0}$ is stationary, see \cite{Robinson1983, Hardle92}. We generalize here the use of the Nadaraya-Watson estimator to the case of an autoregressive process indexed by a binary tree {\it i.e.} a NBAR process.\\

For $n \geq 0$, introduce the genealogical tree up to the $(n+1)$-th generation, $\mathbb T_{n+1} = \bigcup_{m=0}^{n+1} \GG_m$. Assume we observe $\mathbb X^{n+1} = (X_u)_{u \in \mathbb  T_{n+1}}$, {\it i.e.} we have $|\TT_{n+1}| = 2^{n+2}-1$ random variables with value in $\statespace$. 
Let $\Dd \subset \statespace$ be a compact interval. We propose to estimate $(f_0(x), f_1(x))$ the autoregressive functions at point $x \in \Dd$ from the observations $\mathbb X^{n+1}$
 by
\begin{equation}\label{eq:estimators_functions}
\Big( \, \widehat{f}_{\iota, n}(x) = \frac{ |\TT_n|^{-1} \sum\limits_{u \in \TT_n} K_{h_n}(x-X_u) X_{u\iota}}{ \big( |\TT_n|^{-1} \sum\limits_{u \in \TT_n} K_{h_n}(x-X_u) \big) \vee \varpi_n} \, ,  \iota \in \{ 0,1 \} \, \Big) ,
\end{equation}
where $\varpi_n > 0$ and we set $K_{h_n}(\cdot) = h_n^{-1} K(h_n^{-1}\cdot)$ for $h_n>0$ and a kernel function $K : \RR \rightarrow \RR$ such that $\int_\RR K = 1$.

\subsubsection*{Main theoretical results and outline}
Our first objective in this work is to study the estimators \eqref{eq:estimators_functions} both from nonasymptotic and asymptotic points of view. 
To our best knowledge, there is no extensive nonparametric study for nonlinear bifurcating autoregressive processes.
We can mention the applications of Bitseki Penda, Escobar-Bach and Guillin \cite{BPEBG15} (section~4) where deviations inequalities are derived for Nadaraya-Watson type estimators of the autoregressive functions. 
We also refer to Bitseki Penda, Hoffmann and Olivier \cite{BHO} where some characteristics of a NBAR process are estimated nonparametrically (the invariant measure, the mean-transition and the $\TT$-transition).

Our nonasymptotic study includes the control of the quadratic loss in a minimax sense (Theorems~\ref{thm:upper_rateBAR} and~\ref{thm:lower_rateBAR}) and our asymptotic study includes almost sure convergence (Proposition~\ref{prop:asymptotic_as}) and asymptotic normality (Theorems~\ref{thm:asymptotic_normality} and~\ref{thm:asymptotic_normality2}). 
To this end, we shall make use of nonasymptotic behaviour for bifurcating Markov chains (see \cite{Guyon, BDG14}) and asymptotic behaviour of martingales. 
We are also interested in comparing the two autoregressive functions $f_0$ and $f_1$ and to test whether the phenomenon studied is symmetric or not. The test we build (in Theorem \ref{thm:test}) to do so relies on our asymptotic results (see Corollary \ref{cor:TCLBierens}).\\

The present work is organised as follows.
The results are obtained under the assumption of geometric ergodicity of the so-called tagged-branch Markov chain we define in Section~\ref{sec:nonasympto}, together with the nonasymptotic results. 
In Section~\ref{sec:asympto}, we state asymptotic results for our estimators which enable us to address the question of asymmetry and build a test to compare $f_0$ and $f_1$. 
We also give numerical results to illustrate the estimation and test strategies on both simulated and real data (Section~\ref{sec:num}). Section~\ref{sec:discussion} encloses a discussion.
The last parts of the article, Section~\ref{sec:proofs} with an appendix in Section \ref{sec:appendix}, are devoted to the proofs of our results.

\section{\textsc{Nonasymptotic behaviour}} \label{sec:nonasympto}

\subsection{Tagged-branch chain}

In this section we build a chain $(Y_m)_{m \geq 0}$ corresponding to a lineage taken randomly (uniformly at each branching event) in the population, a key object in our study. 
Let $(\iota_m)_{m \geq 1}$ be a sequence of independent and identically distributed random variables that have the Bernoulli distribution of parameter $1/2$. Introduce 
 \begin{equation} \label{eq:gepsmargin}
G_0(\cdot) = \int_\statespace \noisedensity(\cdot,y) dy \quad \text{and} \quad G_1(\cdot) = \int_\statespace \noisedensity(x,\cdot) dx
 \end{equation}
the marginals of the noise density $\noisedensity$ and let $(\varepsilon'_m)_{m \geq 1}$ be a sequence of independent random variables such that $\varepsilon'_m$ given $\iota_m = \iota$ has density $G_\iota$, for $\iota \in \{0,1\}$. In addition $\iota_m$ and $\ep'_m$ are independent of $(Y_k)_{0\leq k \leq m-1}$. We now set
\begin{equation} \label{eq:Y}
\left\{
\begin{array}{l}
  Y_0 = X_{\emptyset},  \\
  Y_{m} = f_{\iota_m}(Y_{m-1}) + \ep'_m \, , \quad m\geq 1.  
\end{array}
\right.
\end{equation}
Then the so-called tagged-branch chain $Y = (Y_m)_{m\geq 0}$ has transition
\begin{equation} \label{eq:QtransBAR}
\Qq(x,dy) = \frac{1}{2}\Big(G_0\big(y-f_0(x)) + G_1\big(y-f_1(x)\big) \Big) dy,
\end{equation}
which means that for all $m\geq1$, we have $\PP(Y_{m} \in dy | Y_{0} = x) = \Qq^{m}(x,dy)$ 
where we set
$$\Qq^{m}(x,dy)=\int_{z\in \Ss}\Qq(x,dz)\Qq^{m-1}(z,dy)  \;\; {\rm with}  \;\;  \Qq^0(x,dy) = \delta_x(dy)$$
for the $m$-th iteration of $\Qq$. 

Asymptotic and nonasymptotic studies have shown that the limit law of the Markov chain $Y$ plays an important role, we refer to Bitseki Penda, Djellout and Guillin \cite{BDG14} and references therein for more details (in the general setting of bifurcating Markov chains). In the present work, the tagged-branch Markov chain will play a crucial role in the analysis of the autoregressive functions estimators\footnote{More precisely, we will see that the denominator of \eqref{eq:estimators_functions} converges almost surely to the invariant density of the Markov chain~$Y$ (Proposition~\ref{prop:nuhat_ps}).}.

\subsection{Model contraints} 
The autoregressive functions $f_0$ and $f_1$ are devoted to belong to the following class.
For $\gamma \in (0,1)$ and $\ell>0$, we introduce the class $\Ff(\gamma,\ell)$ of continuous functions $f : \statespace \rightarrow \statespace$ such that  
$$|f(x)| \leq \gamma |x| + \ell
$$
for any $x\in\statespace$. \\

%
The two marginals $G_0$ and $G_1$ of the noise density are devoted to belong to the following class. For $r>0$ and $\lambda>2$, we introduce the class $\Gg(r,\lambda)$ of nonnegative continuous functions $g:\statespace\rightarrow [0,\infty)$ such that 
$$
g(x) \leq \frac{r}{1+|x|^\lambda} 
$$
for any $x\in\statespace$.
When $(G_0,G_1)\in \Gg(r,\lambda)^2$ for some $\lambda>3$, we denote the covariance matrix of $(\ep_0,\ep_1)$, called noise covariance matrix, by 
\begin{equation}\label{eq:cov}
\Gamma = \begin{pmatrix} \sigma_0^2 & \varrho \sigma_0 \sigma_1 \\ \varrho \sigma_0 \sigma_1  & \sigma_1^2
\end{pmatrix}
\end{equation}
with $\sigma_0\,, \sigma_1>0$ and $\varrho \in (-1,1)$.\\

It is crucial for our proofs to study the ergodicity of the tagged-branch Markov chain $Y$. Geometric ergodicity of nonlinear autoregressive processes has been studied in Bhattacharya and Lee \cite{BL95} (Theorem~1) and also in An and Huang \cite{AH96} and Cline \cite{Cline07}. The main difference is that we need here a precise control on the ergodicity rate,  which should be smaller that $1/2$, due to the binary tree structure. We also see, through \eqref{eq:Y}, that the autoregressive function is random in our case.

The following crucial assumption will guarantee geometric ergodicity of the tagged-branch Markov chain $Y$ with an exponential decay rate smaller than $1/2$ (see Lemma~\ref{lem:ergodicity}).  For any $M>0$, we set
\begin{equation} \label{eq:minorg}
\delta(M) =  \min\Big\{\inf_{|x|\leq M}G_0(x); \inf_{|x|\leq M}G_1(x)\Big\}.
\end{equation}

\begin{assumption}\label{ass:ergodicity2} Set $M_0 = \ell + \EE\big[|\ep_{1}'|\big] < \infty$ where $\ep_{1}'$ has density $(G_0+G_1)/2$.
There exists $\eta >0$ such that $$\gamma<\frac{1}{2}-\eta$$ and there exists $M_1 > 2M_0/(1/2-\eta-\gamma)$ such that
\begin{equation} \label{eq:ergodicity2}
2 M_1 \delta\big((1+\gamma)M_1+\ell\big) > \frac{1}{2}.
\end{equation}
\end{assumption}

The following assumption will guarantee that the invariant density $\nu$ is positive on some nonempty interval (see Lemma~\ref{lem:minorationnu}).  For any $M>0$, we set
$$
\eta(M) = \frac{|G_0|_\infty+|G_1|_\infty}{2} \int_{|y|>M}  \int_{x\in\statespace} \frac{r}{1+\big|y-\gamma |x| -\ell \big|^\lambda \wedge \big| y+\gamma |x|+\ell\big|^\lambda}  dxdy,
$$
where, for a function $h : \statespace\rightarrow \RR$, $|h|_\infty$ stands for $\sup_{x\in\statespace}|h(x)|$.
\begin{assumption}\label{ass:numinor} For $M_2 > 0$ such that $\eta(M_2)<1$, there exists $M_3 > \ell+\gamma M_2$ such that $\delta(M_3)>0$.
\end{assumption}

\subsection{Main results}
We need the following property on $K$:

\begin{assumption} \label{ass:kernel} The kernel $K:\RR \rightarrow \RR$ is bounded with compact support and for some integer $n_0 \geq 1$, we have $\int_{-\infty}^\infty x^kK(x)dx= {\bf 1}_{\{k=0\}}$ for $k=1,\ldots, n_0$.
\end{assumption} 
Assumption~\ref{ass:kernel} will enable us to have nice approximation results over smooth functions $f_0$ and $f_1$, described in the following way: for $\mathcal X \subseteq \mathbb R$ and $\beta>0$, with $\beta=\lfloor \beta\rfloor + \{\beta\}$, $0< \{\beta\} \leq 1$ and $\lfloor \beta\rfloor$ an integer, let $\mathcal H_\mathcal X^\beta$ denote the H\" older space of functions $h:{\mathcal X}\rightarrow \RR$ possessing a derivative of order $\lfloor \beta \rfloor$ that satisfies
\begin{equation} \label{def sob}
|h^{\lfloor \beta \rfloor}(y)-h^{\lfloor \beta \rfloor}(x)| \leq c(h)|x-y|^{\{\b\}}.
\end{equation}
The minimal constant $c(h)$ such that \eqref{def sob} holds defines a semi-norm $|g|_{{\mathcal H}_\mathcal X^\beta}$. We equip the space ${\mathcal H}^\beta_\mathcal X$ with the norm 
$\|h\|_{{\mathcal H}^\beta_\mathcal X} = \sup_{x}|h(x)|+ |h|_{{\mathcal H}_\mathcal X^\beta}$ and the balls
$${\mathcal H}_\mathcal X^\beta(L) = \{h:\mathcal X\rightarrow \RR,\;\|h\|_{{\mathcal H}_\mathcal X^\beta} \leq L\},\;L>0.$$




The notation $\propto$ means proportional to, up to some positive constant independent of $n$ and the notation $\lesssim$ means up to some constant independent of $n$.

\begin{thm}[Upper rate of convergence] \label{thm:upper_rateBAR} 
%
Let $\gamma\in(0,1/2)$ and $\ell>0$, let $r>0$ and $\lambda>3$.
Specify $( \widehat{f}_{0,n}, \widehat{f}_{1,n})$ with a kernel $K$ satisfying Assumption~\ref{ass:kernel} for some $n_0>0$, with
$$
h_n \propto |\TT_n|^{-1/(2\beta+1)}
$$
and $\varpi_n>0$ such that $\varpi_n \rightarrow 0$ as $n\rightarrow\infty$. For every $L,L'>0$ and $0<\beta<n_0$, for every $\noisedensity$ such that $(G_0,G_1)\in \big( \Gg(r,\lambda)\cap \Hh^{\beta}_\statespace(L')\big)^2$ satisfy Assumptions~\ref{ass:ergodicity2} and~\ref{ass:numinor}, there exists $d = d(\gamma,\ell, G_0,G_1)>0$ such that for every compact interval $\Dd \subset [-d,d]$ with nonempty interior and for every $x$ in the interior of $\Dd$,
$$
\sup_{(f_0,f_1)} \EE_\mu \Big[ \big( \widehat{f}_{0,n}(x) - f_0(x)\big)^2 +  \big( \widehat{f}_{1,n}(x) - f_1(x)\big)^2 \Big]^{1/2} \lesssim \varpi_n^{-1} |\TT_n|^{-\beta/(2\beta+1)}
$$
where the supremum is taken among all functions $(f_0,f_1) \in \big( \Ff(\gamma,\ell) \cap \Hh^\b_\Dd(L)\big)^2$, for any initial probability measure $\mu(dx)$ on $\statespace$ for $X_\emptyset$ such that $\mu\big((1+|\cdot|)^2\big)<\infty$.
\end{thm}

Some comments are in order: \\

\noindent {\bf 1)}~The noise density $G$ should be at least as the autoregressive functions $f_0$ and $f_1$, in order to obtain the rate $|\TT_n|^{-\beta/(2\beta+1)}$. Recall now that we build Nadaraya-Watson type estimators: the estimator $\widehat{f}_{\iota,n}$ is the quotient of an estimator of $(\nu f_{\iota})$ and an estimator of $\nu$. Thus the rate of convergence depends not only on the regularity of $f_\iota$ but also on the regularity of $\nu$. Note that $G$ determines the regularity of the mean transition $y \leadsto \Qq(x,y)$ (see \eqref{eq:QtransBAR}) and thus the regularity of the invariant measure $\nu$. A more general result would establish the rate $|\TT_n|^{-\beta'/(2\beta'+1)}$ with $\beta'$ the minimal regularity between $f_\iota$ and $\nu$ (and thus between $f_\iota$ and $G$ the noise density).
\noindent {\bf 2)}~The autoregressive functions $f_0$ et $f_1$ should be locally smooth, on $\Dd$ a vicinity of $x$ (and not necessarily globally smooth, on $\RR$). Note that we could also state an upper bound for $\EE_\mu \big[ \int_\Dd \big( \widehat{f}_{0,n}(x) - f_0(x)\big)^2 +  \big( \widehat{f}_{1,n}(x) - f_1(x)\big)^2 dx \big]$.
\noindent {\bf 3)}~Up to the factor $\varpi_n^{-1}$, we obtain the classical rate $|\TT_n|^{-\beta/(2\beta+1)}$ where $|\TT_n|$ is the number of observed couples $(X_u,X_{u\iota})$. We know it is optimal in a minimax sense in a density estimation framework and we can infer this is optimal in our framework too. To prove it is the purpose of Theorem~\ref{thm:lower_rateBAR} which follows.
%
%
\noindent  {\bf 4)}~We do not achieve adaptivity in the smoothness of the autoregressive functions since our choice of bandwidth $h_n$ still depends on $\beta$. For classical autoregressive models ({\it i.e.} nonbifurcating), adaptivity is achieved in a general framework in the early work by Hoffmann \cite{H6}, and we also refer to Delouille and van Sachs  \cite{Delouille}. 
\noindent {\bf 5)}~For the sake of simplicity, we have picked a common bandwidth $h_n$ to define the two estimators, but one can immediately generalize our study for two different bandwidths $(h_{\iota,n} = |\TT_n|^{-\beta_\iota/(2\beta_\iota+1)},\iota\in\{0,1\})$ where $\beta_\iota$ is the H\"older smoothness of $f_\iota$.\\

We complete the previous theorem with
\begin{thm}[Lower rate of convergence] \label{thm:lower_rateBAR} Assume the noise density $\noisedensity$ is a bivariate Gaussian density. 
Let $\Dd\subset \statespace$ be a compact interval. 
For every $\gamma\in (0,1)$ and every positive $\ell, \beta, L$, there exists $C>0$ such that, for every $x\in\Dd$, 
$$
\liminf_{n\rightarrow \infty} \inf_{(\widehat{f}_{0,n},  \widehat{f}_{1,n})} \, \sup_{(f_0,f_1)} \PP \Big(  |\TT_n|^{\beta/(2\beta+1)}   \big(\big|\widehat{f}_{0,n}(x) - f_0(x)\big| +  \big| \widehat{f}_{1,n}(x) - f_1(x)\big|  \big) \geq C \Big) > 0,
$$
where the supremum is taken among all functions $(f_0,f_1) \in \big( \Ff(\gamma,\ell) \cap \Hh^\b_\Dd(L) \big)^2$ and the infimum is taken among all estimators based on $(X_u,u\in\TT_{n+1})$.
\end{thm}
This result obviously implies a lower rate of convergence for the mean quadratic loss at point~$x$. We see that in a regular case, the Gaussian case, the lower and upper rates match.

\section{\textsc{Asymmetry test}} \label{sec:asympto}
Testing in the context of nonparametric regression is a crucial point, especially in applied contexts. 
The question of no effect in nonparametric regression is early addressed in Eubank and LaRiccia \cite{Eubank93}. 
We may also want to compare two regression curves nonparametrically and we refer to Munk \cite{Munk98} and references therein. 
Specific tools have been developed to compare time series 
(see for instance the recent work by Jin \cite{Jin11} among many others).
The test we propose to study asymmetry in a NBAR process is based on the following asymptotic study.

\subsection{Preliminaries: asymptotic behaviour} \label{sec:asymptoticstudy}
%
The almost-sure convergence of the autoregressive functions estimators is obtained in Proposition~\ref{prop:asymptotic_as} for any $h_n \propto |\TT_n|^{-\a}$ with $\a \in (0,1)$. Choosing $\a\geq 1/(2\beta+1)$, the estimator $\big(\widehat{f}_{0,n}(x), \widehat{f}_{1,n}(x)\big)$ recentered by $\big(f_0(x),f_1(x)\big)$ and normalised by $\sqrt{|\TT_n| h_n}$ converges in distribution to a Gaussian law. Depending on $\a > 1/(2\beta+1)$ or not, the limit Gaussian law is centered or not, as we state in Theorems~\ref{thm:asymptotic_normality} and~\ref{thm:asymptotic_normality2}.
\begin{prop}[Almost sure convergence] \label{prop:asymptotic_as}
In the same setting as in Theorem~\ref{thm:upper_rateBAR} with $h_n \propto |\TT_n|^{-\alpha}$ for $\a \in (0,1)$,
\begin{equation*}
\begin{pmatrix}\widehat{f}_{0,n}(x)\\ \widehat{f}_{1,n}(x)\end{pmatrix} \rightarrow \begin{pmatrix}f_{0}(x)\\ f_{1}(x)\end{pmatrix} , \quad \PP_\mu - a.s.
\end{equation*}
as $n\rightarrow + \infty$.
\end{prop}

From now on we need to reinforce the assumption on the noise sequence: we require that the noise $(\varepsilon_{0},\varepsilon_{1})$ has finite moment of order $4$, $\EE\big[\varepsilon_0^4 + \varepsilon_1^4\big]  < \infty$, which is guaranteed by $G\in\Gg(r,\lambda)$ for $\lambda>5$. We use the notation $|K|_2^2 = \int_\statespace K(x)^2 dx$.

\begin{thm}[Asymptotic normality] \label{thm:asymptotic_normality}
In the same setting as in Theorem~\ref{thm:upper_rateBAR} with $\lambda > 5$ and $h_n \propto |\TT_n|^{-\alpha}$ for $\alpha\in \big(1/(2\beta+1),1\big)$,
\begin{equation*}
\sqrt{|\TT_n|h_n}
\begin{pmatrix} \widehat{f}_{0,n}(x)-f_{0}(x)\\ \widehat{f}_{1,n}(x)-f_{1}(x)\end{pmatrix} 
\toL
\Nn_{2} \big( \boldsymbol{0}_2  , \boldsymbol{\Sigma}_2(x) \big) \quad \text{ with } \quad \boldsymbol{\Sigma}_2(x) = | K |_2^2  \big(\nu(x) \big)^{-1} \Gamma,
\end{equation*}
$\Gamma$ being the noise covariance matrix and $\boldsymbol{0}_2 = (0,0)$.
Moreover, for $x_1,\ldots, x_k$ distinct points in $\Dd$, the sequence 
$$
\left( \begin{pmatrix} \widehat{f}_{0,n}(x_l)-f_{0}(x_l)\\ \widehat{f}_{1,n}(x_l)-f_{1}(x_l)\end{pmatrix}  , \, l = 1, \ldots, k \right) 
$$
is asymptotically independent.
\end{thm}
The restriction $\alpha>1/(2\beta+1)$ in Theorem~\ref{thm:asymptotic_normality} prevents us from choosing $h_n \propto |\TT_n|^{-1/(2\beta+1)}$, which is the optimal choice to achieve the minimax rate as we have seen in Theorem~\ref{thm:upper_rateBAR}.
The following Theorem remedies to this flaw, but at the cost of an unknown bias. We obtain an explicit expression of this bias for $\beta$ an integer which depends on the $\beta$-th derivatives of the autoregressive functions and the invariant measure of the tagged-branch chain.

\begin{thm}[Asymptotic normality with bias expression] \label{thm:asymptotic_normality2}
In the same setting as in Theorem~\ref{thm:upper_rateBAR} with $\lambda > 5$ and $\beta$ an integer,
\begin{itemize}
\item[(i)] If $h_n^{\beta} \sqrt{|\TT_n| h_n} \rightarrow \kappa$ with $\kappa \in [0,\infty)$ as $n \rightarrow + \infty$, then
\begin{equation*}
\sqrt{|\TT_n|h_n}
\begin{pmatrix} \widehat{f}_{0,n}(x)-f_{0}(x)\\ \widehat{f}_{1,n}(x)-f_{1}(x)\end{pmatrix} 
\toL
\Nn_{2}\left( \kappa \boldsymbol{m}_2(x) , \boldsymbol{\Sigma}_2(x) \right) \quad \text{with } \quad \boldsymbol{\Sigma}_2(x) =  | K |_2^2  \big(\nu(x) \big)^{-1} \Gamma
\end{equation*}
and
$$
\boldsymbol{m}_2(x) = \frac{(-1)^{\beta}}{\beta! \, \nu(x)}  \int_\statespace y^{\beta} K(y) dy \begin{pmatrix} (\nu f_0)^{\beta}(x) - \nu^{\beta}(x)f_0(x) \\ (\nu f_1)^{\beta}(x) - \nu^{\beta}(x) f_1(x)   \end{pmatrix}.
$$
\item[(ii)] If $h_n^{\beta} \sqrt{|\TT_n|h_n} \rightarrow + \infty$ as $n \rightarrow + \infty$, then
$$
h_n^{-\b} \begin{pmatrix} \widehat{f}_{0,n}(x)-f_{0}(x)\\ \widehat{f}_{1,n}(x)-f_{1}(x)\end{pmatrix}  \overset{\PP_\mu}{\longrightarrow}  \boldsymbol{m}_2(x).
$$
\end{itemize}
\end{thm}

\subsection{Construction of an asymmetry test}

Let $x_1,\ldots, x_k$ be distinct points in $\statespace$.
We are going to build a statistical test that allows us to segregate between hypothesis
$$
\Hh_0 : \forall l \in \{ 1, \ldots, k\}, f_0(x_l) = f_1(x_l) \quad \text{vs.} \quad \Hh_1 : \exists l \in\{ 1, \ldots, k\}, f_0(x_l) \neq f_1(x_l).
$$
In the parametric studies on \textit{E. coli} \cite{Guyon, DM10}, these tests are known as detection of cellular aging and they permit to decide if the cell division is symmetric or asymmetric.

\subsubsection*{Construction of an asymptotically unbiased estimator} Inspired by Bierens \cite{Bierens} we define new estimators in order to both achieve the rate $|\TT_n|^{-\beta/(2\beta+1)}$ in the asymptotic normality property and remove the asymptotic bias. Let $\big(\widehat{f}_{\iota,n}^{(a)}(x),\iota\in\{0,1\}\big)$ be the estimators \eqref{eq:estimators_functions} with bandwidth $h_n^{(a)} \propto |\TT_n|^{-1/(2\beta+1)}$ and $\big(\widehat{f}_{\iota,n}^{(b)}(x),\iota\in\{0,1\}\big)$ be the estimators \eqref{eq:estimators_functions} with bandwidth $h_n^{(b)} \propto |\TT_n|^{-\delta/(2\beta+1)}$ for some $\delta\in (0,1)$. 
We define
\begin{equation} \label{eq:Bierensestimator}
\Big( \bar{f}_{\iota,n}(x) = \big( 1-|\TT_n|^{\frac{-\beta(1-\delta)}{2\b+1}} \big)^{-1} \big( \widehat{f}_{\iota,n}^{(a)}(x) - |\TT_n|^{\frac{-\beta(1-\delta)}{2\b+1}} \widehat{f}_{\iota,n}^{(b)}(x)\big), \iota\in\{0,1\} \Big).
\end{equation}
%
%
\begin{cor} \label{cor:TCLBierens} In the same setting as in Theorem~\ref{thm:asymptotic_normality2},
\begin{equation*} \label{eq:TCLBierens}
|\TT_n|^{\frac{\b}{2\b+1}}
\begin{pmatrix} \bar{f}_{0,n}(x)-f_{0}(x)\\ \bar{f}_{1,n}(x)-f_{1}(x)\end{pmatrix} 
\toL
\Nn_{2} \big( \boldsymbol{0}_2  , \boldsymbol{\Sigma}_2(x) \big) \quad \text{ with } \quad \boldsymbol{\Sigma}_2(x) = | K |_2^2  \big(\nu(x) \big)^{-1} \Gamma,
\end{equation*}
for every $\delta \in (0,1)$ in the definition \eqref{eq:Bierensestimator} of $\big( \bar{f}_{0,n}(x), \bar{f}_{1,n}(x)\big)$.
\end{cor}
As announced the trick of Bierens \cite{Bierens} enables us to remove the unknown asymptotic bias while keeping the optimal rate of convergence.

\subsubsection*{Test statistics}  
We define a test statistics based on the new estimators $\big(\bar f_{0,n},\bar f_{1,n}\big)$ by
\begin{equation}\label{eq:test_stat}
W_n(x_1,\ldots,x_k) = \frac{|\TT_n|^{\frac{2\b}{2\b+1}}}{(\sigma_0^2+ \sigma_1^2-2 \sigma_0 \sigma_1\varrho )| K |_2^2} \sum_{l = 1}^k \widehat{\nu}_n(x_l) \big(\bar{f}_{0,n}(x_l) - \bar{f}_{1,n}(x_l)\big)^2
\end{equation}
with $\widehat{\nu}_n(\cdot) = |\TT_n|^{-1} \sum_{u\in \TT_n} K_{h_n}(X_u-\cdot)$
where $h_n \propto |\TT_n|^{-1/(2\beta+1)}$.
\begin{thm}[Wald test for asymmetry] \label{thm:test}
In the same setting as in Theorem~\ref{thm:asymptotic_normality2}, 
let $x_1,\ldots, x_k$ be distinct points in $\Dd$.  Then the test statistic $W_n(x_1,\ldots,x_k)$ converges in distribution to the chi-squared distribution with k degrees of freedom $\chi^{2}(k)$, under $\Hh_{0}$, and $\PP_\mu$-almost surely to $+\infty$, under $\Hh_1$.
\end{thm}

Note that in \eqref{eq:test_stat} we could replace $\sigma_0^2$, $\sigma_1^2$ and $\varrho$ by 
\begin{equation} \label{eq:varest}
\widehat{\sigma}_{0,n}^2 = |\TT_n|^{-1} \sum_{u \in \TT_n}  (\widehat{\varepsilon}_{u0})^2 \, , \quad \widehat{\sigma}_{1,n}^2 = |\TT_n|^{-1} \sum_{u \in \TT_n}  (\widehat{\varepsilon}_{u1})^2 \, ,\quad \widehat{\varrho}_n =  \widehat{\sigma}_{0,n}^{-1} \widehat{\sigma}_{1,n}^{-1} |\TT_n|^{-1} \sum_{u \in \TT_n}  \widehat{\varepsilon}_{u0} \widehat{\varepsilon}_{u1}
\end{equation}
with the empirical residuals
$ \widehat{\varepsilon}_{u\iota} =   X_{u\iota} - \widehat{f}_{\iota,n}(X_u)$ for  $u \in \TT_n.$ We claim that these estimators are consistent, so that Theorem~\ref{thm:test} is still valid for this new test statistics. Proving the convergence in probability of these three quantities would imply some technical calculations and we do not give here more details.


\section{\textsc{Numerical implementation}} \label{sec:num}

\subsection{Simulated data}

The goal of this section is to illustrate the theoretical results of the previous sections, in particular the results of Theorem~\ref{thm:upper_rateBAR} (Upper rate of convergence) and Theorem~\ref{thm:test} (Wald test for asymmetry).

\subsubsection*{Quality of the estimation procedure.} 
We pick trial autoregressive functions defined analytically by
$$
f_0(x) = x\big(1/4 + \exp(-x^{2})/2\big) \quad \text{and} \quad f_1(x) = x\big(1/8 + \exp(-x^{2})/2\big)
$$ 
for $x \in \statespace$. We take a Gaussian noise with $\sigma_0^2 = \sigma_1^2 = 1$ and $\varrho = 0.3$. 
We simulate $M = 500$ times a NBAR process up to generation $n + 1 = 15$, with root $X_\emptyset = 1$. 
We take a Gaussian kernel $K(x)= (2\pi)^{-1/2}\exp(-x^2/2)$ and $h_n = |\TT_n|^{-1/5}$ in order to implement estimators given by~\eqref{eq:estimators_functions}. We evaluate $\widehat{f}_{0,n}$ and $\widehat{f}_{1,n}$ on a regular grid of $\mathcal{D} = [-3,3]$ with mesh $\Delta x = |\TT_n|^{-1/2}$. We did not meet any problem with the denominator in practice and actually set $\varpi_n = 0$.  For each sample we compute the empirical error 
\begin{equation*} \label{error def}
e_\iota^{(i)} = \frac{\|\widehat{f}^{(i)}_{\iota,n} - f_\iota\|_{\Delta x}}{\|f_\iota\|_{\Delta x}},\quad i=1,\ldots, M,
\end{equation*}
where $\|\cdot\|_{\Delta x}$ denotes the discrete norm over the numerical sampling. Table~\ref{tab:empiricalerror_BAR} displays the mean-empirical errors together with the empirical standard deviations,
$\overline{e}_\iota =M^{-1}\sum_{i=1}^M e_\iota^{(i)}$ 
and
$\big( M^{-1}\sum_{i=1}^M (e_\iota^{(i)}-\overline{e}_\iota)^2 \big)^{1/2}$
 for $\iota\in\{0,1\}.$
The larger $n$, the better the reconstruction of $f_0$ and $f_1$ as shown in Table~\ref{tab:empiricalerror_BAR}.



\begin{table}[h!]
\begin{center}
\begin{tabular}{ccccccccc}
\hline \hline
$\boldsymbol{n}$  &  $8$ & $9$ & $10$ & $11$ & $12$ & $13$ & $14$
\\
$\boldsymbol{|\TT_n|}$  & 511 & 1~023 & 2~047 & 4~095 & 8~191 & 16~383 & 32~767 \\
\hline
$\boldsymbol{\overline{e}_0}$ & 0.4442 & 0.3417 & 0.2633 & 0.2006 & 0.1517 & 0.1285 & 0.0891 \\
\textit{sd.} & \textit{\small{0.1509}} & \textit{\small{0.1063}} & \textit{\small{0.0761}} & \textit{\small{0.0558}} & \textit{\small{0.0387}} & \textit{\small{0.0295}} & \textit{\small{0.0209}} \\
\hline
$\boldsymbol{\overline{e}_1}$ & 0.6696 & 0.5141 & 0.4006 & 0.3027 & 0.2356 & 0.1776 &  0.1384 \\
\textit{sd.} & \textit{\small{0.2482}} & \textit{\small{0.1626}} & \textit{\small{0.1227}} & \textit{\small{0.0831}} & \textit{\small{0.0622}} & \textit{\small{0.0440}} & \textit{\small{0.0326}} \\
\hline \hline
\end{tabular}
\end{center}
\caption{[Simulated data] {\it Mean empirical relative error $\overline{e}_0$ (resp. $\overline{e}_1$) and its standard deviation computed over $M=500$ Monte-Carlo trees, with respect to $|\TT_n|$, for the autoregressive function $f_0$ (resp. $f_1$) reconstructed over the interval $\mathcal{D} = [-3,3]$ by the estimator $\widehat{f}_{0,n}$ (resp. $\widehat{f}_{1,n}$).} \label{tab:empiricalerror_BAR}
}
\end{table}


 This is also true at a visual level, as shown on Figure~\ref{fig:IC_BAR} where $95\%$-level confidence  bands are built so that for each point $x$, the lower and upper bounds include $95\%$ of the estimators $(\widehat{f}_{0,n}^{(i)}(x), i = 1\ldots M)$. As one can see on Figure~\ref{fig:IC_BAR}, the reconstruction is good around $0$ and deteriorates for large or small~$x$.  Indeed the invariant measure estimator shows that its mass is located around $0$ and thus more observations lie in this zone, which ensures a better reconstruction there. The same analysis holds for the reconstruction of $f_1$, see the thin blue lines. \\

\begin{figure}[h!]
\includegraphics[width=7.4cm]{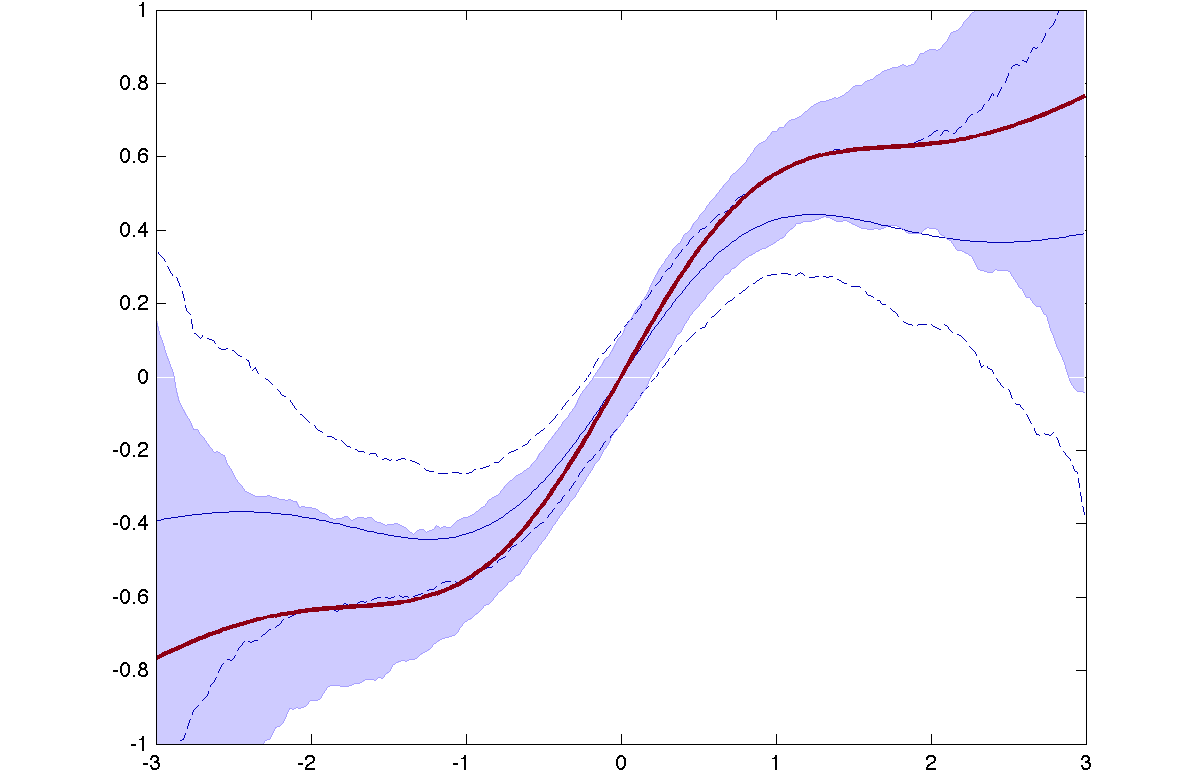}
\includegraphics[width=7.4cm]{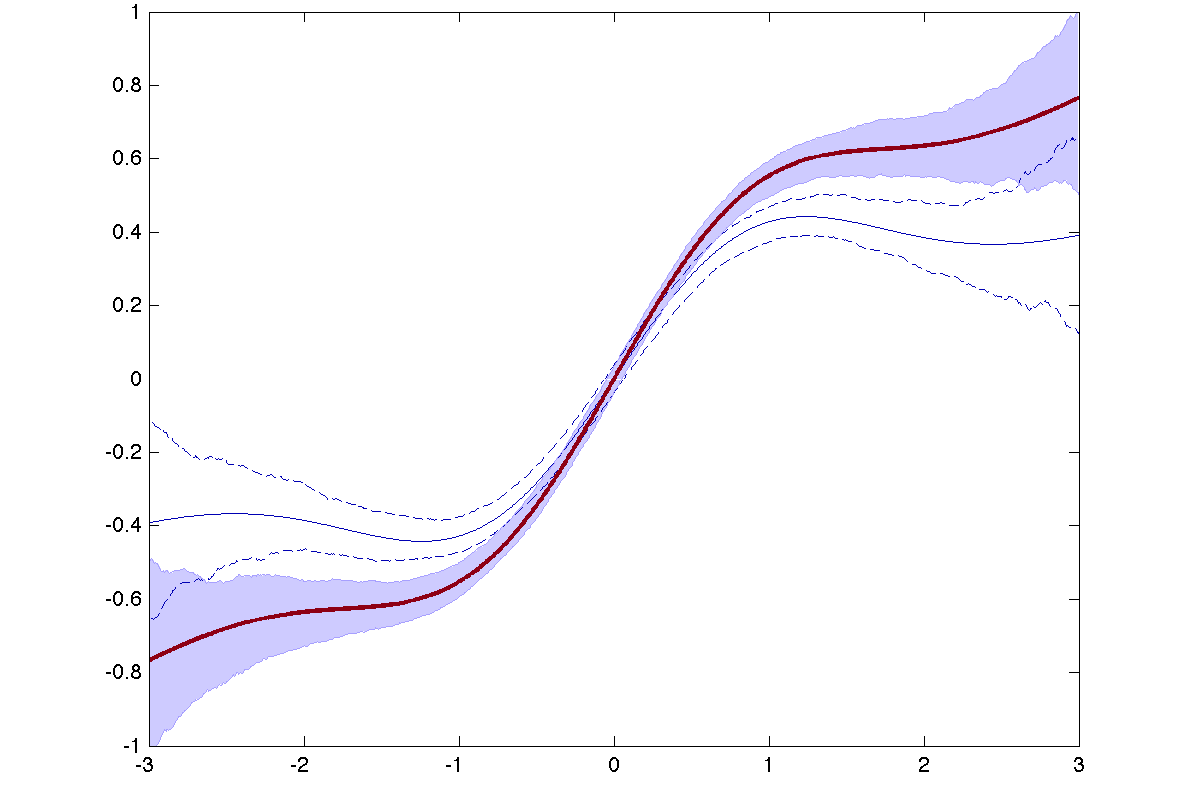}
\caption{[Simulated data]  {\it Reconstruction of $f_0$ over $\mathcal{D}= [-3,3]$ with $95\%$-level confidence bands constructed over $M = 500$ Monte Carlo trees. In bold red line: $x \leadsto f_0(x)$. In thin blue lines: reconstruction of $f_1$ with $95\%$-level confidence bands. Left: $n=10$ generations. Right: $n=14$ generations.} \label{fig:IC_BAR}} 
\end{figure}

The error is close to $|\TT|^{-2/5}$ for both $\widehat{f}_{0,n}$ and $\widehat{f}_{1,n}$ as expected: indeed, for a kernel of order $n_0$, the bias term in density estimation is of order $h^{\beta\wedge(n_0+1)}$. For the smooth $f_0$, $f_1$ and $\nu$ we have here, we rather expect for the rate $|\TT_n|^{-(n_0+1)/(2(n_0+1)+1)} = |\TT|^{-2/5}$ for the Gaussian kernel with $n_0=1$ that we use here, and this is consistent with what we observe on Figure~\ref{fig:erreurBAR}.\\

\begin{figure}[h!]
\centering
\includegraphics[width=8cm]{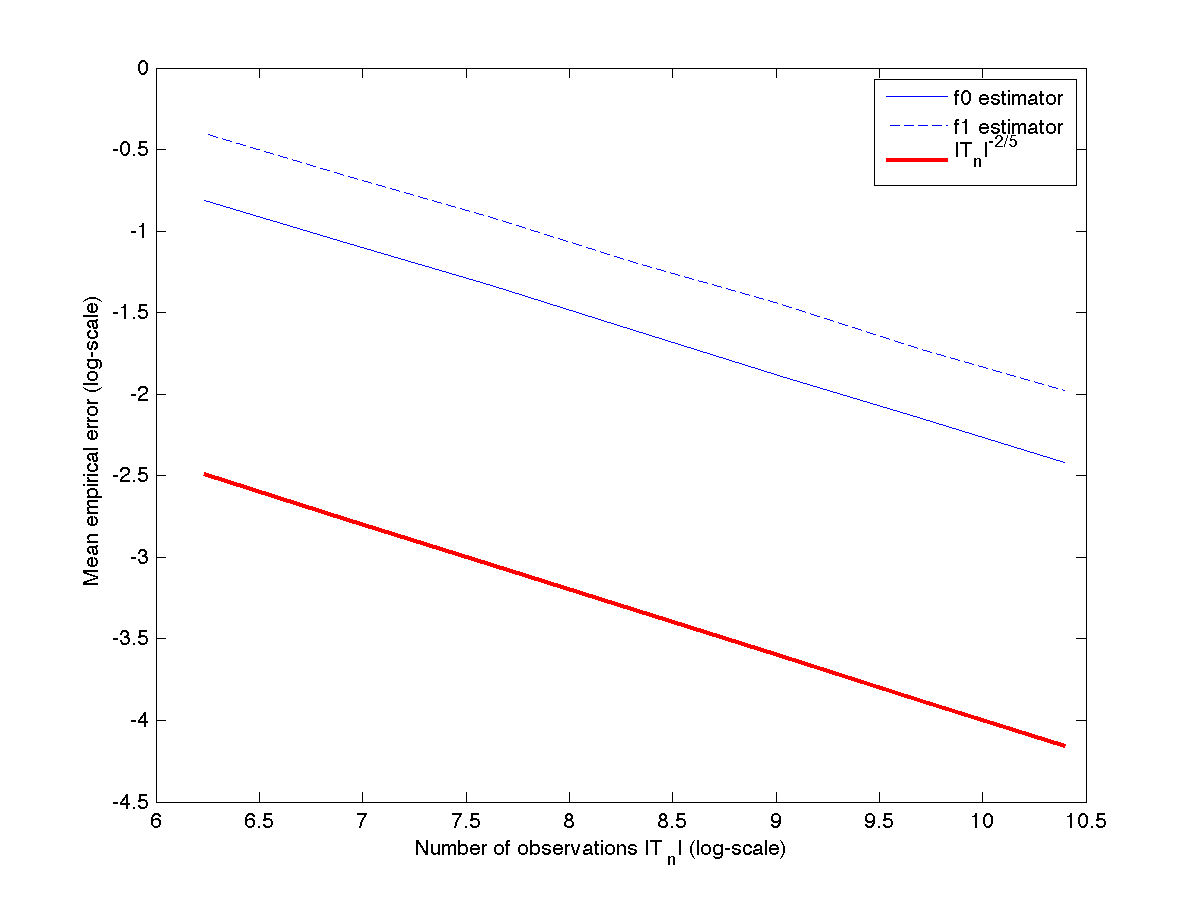}
\caption{[Simulated data]  {\it The log-average relative empirical error over $M = 500$ Monte Carlo trees vs. $\log(|\TT_n|)$ for $f_0$ (resp. $f_1$) reconstructed over $\mathcal{D} = [-3,3]$ by $\widehat{f}_{0,n}$ (solid blue line) (resp. $\widehat{f}_{1,n}$ (dashed blue line)) compared to the expected log-rate (solid red line). \label{fig:erreurBAR}}}
\end{figure}

\subsubsection*{Implementation of the asymmetry test.} 
We implement now the estimators \eqref{eq:Bierensestimator} inspired by \cite{Bierens} in order to compute our test statistics \eqref{eq:test_stat}. We keep a Gaussian kernel and we pick $h_n^{(a)} = |\TT_n|^{-1/5}$ and $h_n^{(b)} = |\TT_n|^{-1/10}$ (\textit{i.e.} $\delta = 1/2$ -- and the choice of $\delta$ proves to have no influence on our numerical results). The numerical study of $\bar{f}_{0,n}$ and $\bar{f}_{1,n}$ leads to similar results as those of the previous study. 
For a given grid $\{x_1,\ldots,x_k\}$ of $\Dd = [-3,3]$, we reject the null hypothesis $\Hh_0$ if $W_n(x_1,\ldots,x_k)$ exceeds the 5\%-quantile of the chi-squared distribution with $k$ degrees of freedom and thus obtain a test with asymptotic level 5\%. We measure the quality of our test procedure computing the proportion of rejections of the null. \\

We first implement the two following cases:
\begin{align*}
(\text{{\bf Case}} \boldsymbol{=}) \quad & f_0(x) = f_1(x) = x(1/4 + \exp(-x^{2})/2), \\
(\text{{\bf Case}} \boldsymbol{\neq}) \quad & f_0(x) = x(1/4 + \exp(-x^{2})/2) \quad \text{and} \quad f_1(x) = x(1/8 + \exp(-x^{2})/2).
\end{align*}
The test should reject $\Hh_0$ in the second case but not in the first one. The larger $n$, the better the test as one can see in Table~\ref{tab:test_BAR}: $\Hh_0$ is more and more often rejected for $(\text{{\bf Case}} \boldsymbol{\neq})$ and less and less often rejected for $(\text{{\bf Case}} \boldsymbol{=})$ as $n$ increases, which is what we expect. We also observe the influence of the number of points of the grid which enables us to build the test statistics. Three grids of $\Dd=[-3,3]$ are tested with $k=13, 25$ and 61 points. 
The larger the number of points, the larger the proportion of rejections of $\Hh_0$ in both cases. However, for $(\text{{\bf Case}} \boldsymbol{=})$, more that $n= 14$ generations are needed to reach the theoretical asymptotic level of 5\%. The choice of the number $k$ of points is delicate but we would recommend to use a low $k$ to build the test. 



\begin{table}[h]
\begin{center}
\begin{tabular}{clccccccc}
\hline \hline
$\boldsymbol{n}$                   & \multicolumn{1}{l}{} & 8      & 9       & 10      & 11      & 12      & 13     & 14     \\
$\boldsymbol{|\TT_n|}$             & \multicolumn{1}{l}{} & 511 & 1~023 & 2~047 & 4~095 & 8~191 & 16~383 & 32~767      \\
\hline
                                   & $\Delta x = 0.5$  &   46.8\% & 67.2\% & 87.6\%  & 99.0\%    & 100\% &    100\% & 100\%      \\
{\bf Case} $\boldsymbol{\neq}$ & $\Delta x = 0.25$    & 59.6\% & 77.8\%  & 92.8\%  & 99.8\%  & 100\% & 100\% & 100\% \\
                                   & $\Delta x = 0.1$     & 67.8\% & 85.4\%  & 95.6\%  & 99.8\%  & 100\% & 100\% & 100\% \\
\hline
                                   & $\Delta x = 0.5$     & 19.6\% & 18.6\%  & 18.2\%  & 16.2\%  & 13.4\%  & 14.8\% & 12.4\% \\
{\bf Case} $\boldsymbol{=}$   & $\Delta x = 0.25$    & 30.4\% & 30.0\%    & 29.0\%    & 24.8\%  & 21.4\%  & 19.4\% & 19.8\% \\
\multicolumn{1}{l}{}               & $\Delta x = 0.1$     & 42.6\% & 42.6\%  & 40.4\%  & 39.8\%  & 35\%    & 31.6\% & 32.2\% \\
\hline \hline
\end{tabular}
\end{center}
\caption{[Simulated data]  {\it Proportions of rejections of the null hypothesis $\Hh_0:$ $\{\forall l = 1, \ldots, k, f_0(x_l) = f_1(x_l)\}$ for 5\% asymptotic level tests over $M=500$ Monte-Carlo trees. The test is based on the test statistics $W_n(x_1,\ldots,x_k)$ \eqref{eq:test_stat} with the grids $\{x_l =-3 + (l-1) \Delta x \leq 3; l \geq 1\}$ for $\Delta x \in \{0.5;0.25;0.1\}$ ({\it i.e.} $k=13, 25$ and 61 points).  {\rm (}{\bf Case} $\boldsymbol{\neq}${\rm ):} the proportions (power of the test) should be high. {\rm (}{\bf Case} $\boldsymbol{=}${\rm ):} the proportions (type I error) should be low.} \label{tab:test_BAR}
}
\end{table}




\begin{figure}[h!]
\includegraphics[width=7.4cm]{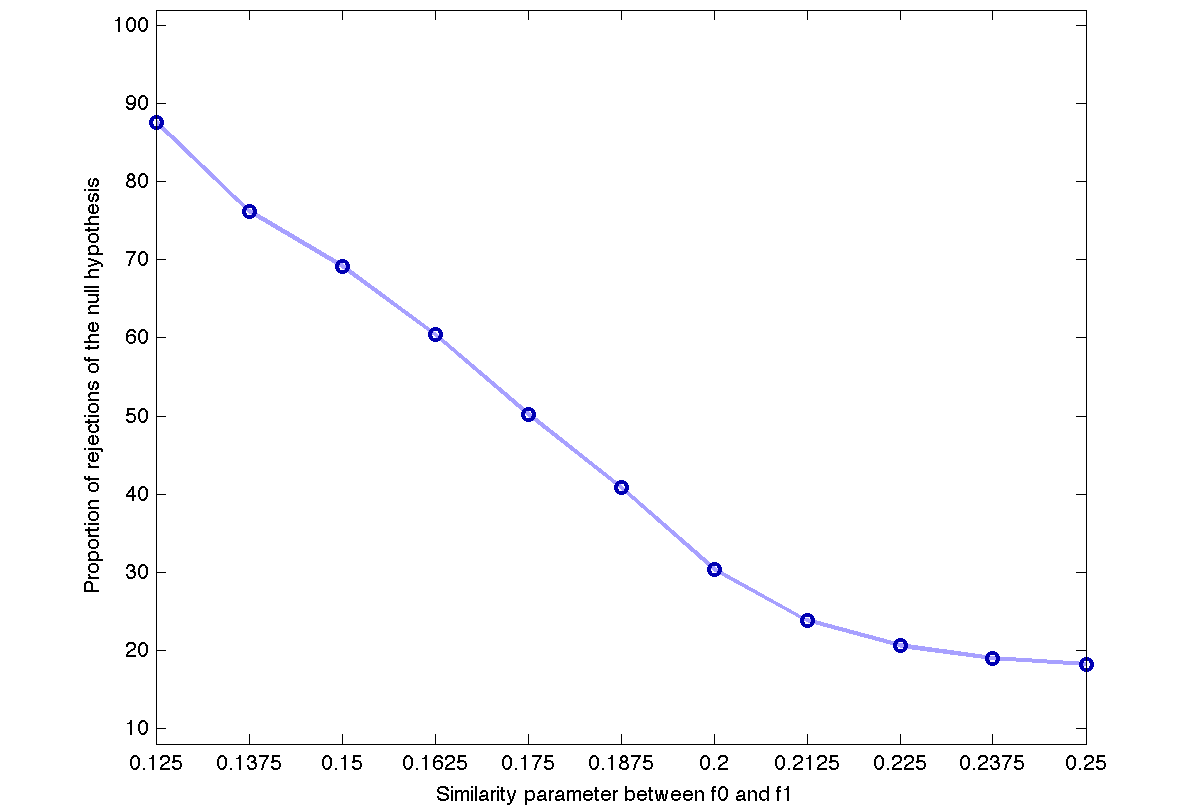}
\includegraphics[width=7.4cm]{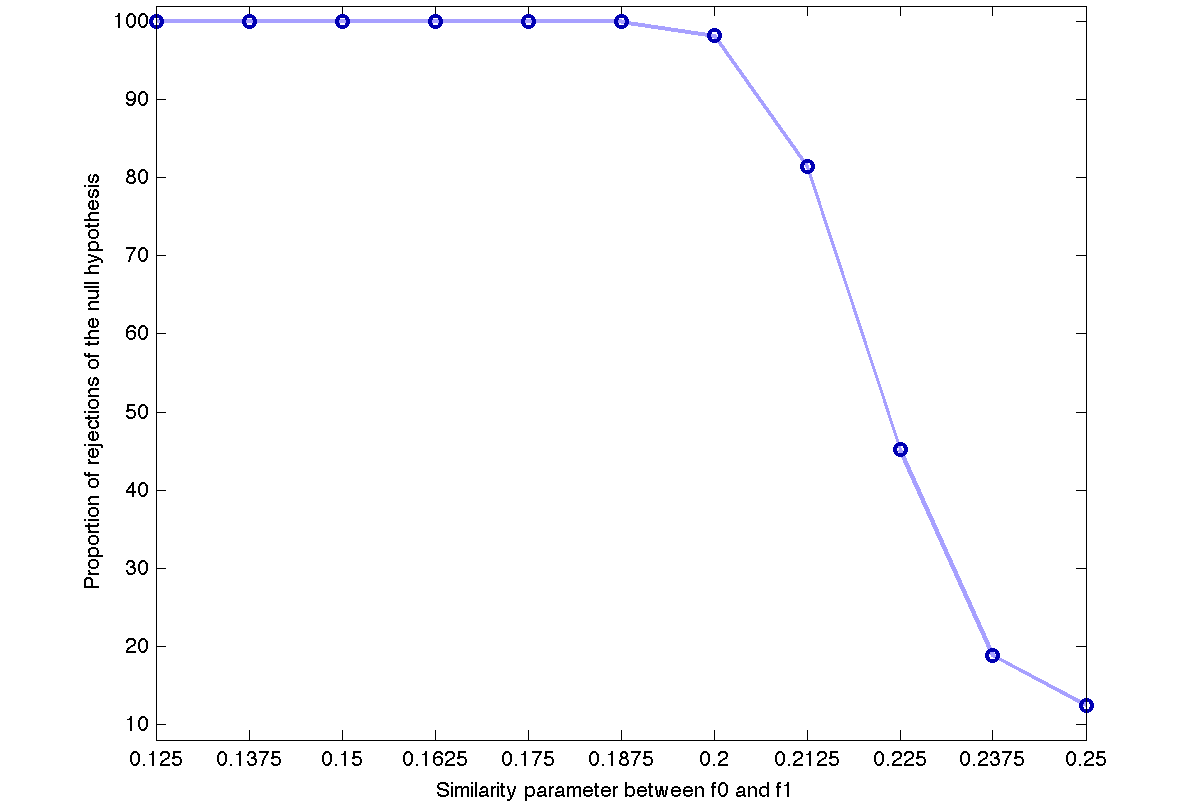}
\caption{[Simulated data]  {\it Proportions of rejections of the null hypothesis $\Hh_0$ (power of the test): $\{\forall l \in\{ 1, \ldots, k\}, f_0(x_l) = f_{1,\tau}(x_l)\}$ with respect to $\tau \in [1/8;1/4]$ for 5\% asymptotic level tests over $M=500$ Monte-Carlo trees. 
The test is based on the test statistics $W_n(x_1,\ldots,x_k)$ \eqref{eq:test_stat} with the grid $\{x_l =-3 + (l-1) \Delta x \leq 3; l \geq 1\}$ for $\Delta x = 0.5$ ({\it i.e.} $k=13$ points). 
Left: $n=10$ generations. 
Right: $n=14$ generations.} \label{fig:puissance_BAR}}
\end{figure}

The second numerical test aims at studying empirically the power of our test. We keep with the same autoregressive function $f_0$ for cells of type 0 and parametrize the autoregressive function for cells of type 1 such that it progressively comes closer to $f_0$:
$$
f_0(x) = x(1/4 + \exp(-x^{2})/2) \quad \text{and} \quad f_{1,\tau}(x) = x(\tau + \exp(-x^{2})/2)
$$
for $\tau\in[1/8,1/4]$.
This choice enables us to interpolate between ({\bf Case} $\boldsymbol{\neq}$) and ({\bf Case} $\boldsymbol{=}$). 

As $\tau$ becomes closer to $1/4$, \textit{i.e.} as $f_{1,\tau}$ becomes closer to $f_0$, we see the decrease of the proportions of rejections of the null on Figure~\ref{fig:puissance_BAR}. The steeper the decrease is, the better performs our test. The proportion of rejections of $\Hh_0$ is higher than 40\% only for $\tau$ up to 0.1875 for a reasonable number of observations ($|\TT_n| = 2~047$ on the left in Figure~\ref{fig:puissance_BAR}). On the right in Figure~\ref{fig:puissance_BAR}, one can see what become the results for a larger number of observations, $|\TT_n| = 32~767$: the performance is good for $\tau$ up to $0.225$, \textit{i.e.} closer to the equality case $\tau = 1/4$.


\subsection{Real data}

Quoting Stewart {\it et al.} \cite{SMPT05}, "The bacterium {\it E. coli} grows in the form of a rod, which reproduces by dividing in the middle. (...) one of the ends of each cell has just been created during division (termed the new pole), and one is pre-existing from a previous division (termed the old pole)."  At each division, the cell inheriting the old pole of the progenitor cell is of type 1, say, while the cell inheriting the new pole is of type 0. 
 The individual feature we focus on is the growth rate ({\it E. coli} is an exponential growing cell). Stewart {\it et al.} \cite{SMPT05} followed the growth of a large number of micro-colonies of {\it E. coli}, measuring the growth rate of each cell up to 9 generations (possibly with some missing data).

Recently, concerning the data set of Stewart {\it et al.},  Delyon {\it et al.} \cite{SaportaKrell} found out that "There is no stationarity of the growth rate across generations. This means that the initial stress of the experiment has not the time to vanish during only the first 9 generations." As a consequence, we should not aggregate the data from the different micro-colonies and we shall only work on the largest samples. The largest genealogical tree counts $655$ cells for which we know both the type and the growth rate, together with the type and the growth rate of the progenitor. The autoregressive functions estimators $(\bar f_{0,n}, \bar f_{1,n})$ are represented on Figure~\ref{fig:Stew655}\footnote{We keep on with a Gaussian kernel, the bandwidths are picked proportional to $N^{-1/5}$ and $N^{-1/10}$ ({\it i.e.} $\delta = 1/2$) with $N = 655$, up to a constant fixed using the rule of Silverman. We underline we do not observe the full tree $\TT_9$, and we compute our estimators accordingly. Point-wise confidence intervals of asymptotic level 95\% built using Corollary \ref{cor:TCLBierens} overlap the curves $x\leadsto \bar f_{0,n}(x)$ and $x\leadsto \bar f_{1,n}(x)$, and are far too optimistic (since $n \leq 9$).}. It is surprising that the relationship between the growth rates of the mother and its offspring may be decreasing.  In this case, our nonparametric estimated curves are close to the linear estimated curves (computed using the estimators of Guyon \cite{Guyon}). We show a second example on Figure \ref{fig:Stew446} where the relationship may not be linear.

Previously, Guyon {\it et al.} \cite{proceedingaging} and de Saporta {\it et al.} \cite{dSGPM12, dSGPM14} carried out asymmetry tests, and our conclusions seem to coincide with the previous ones. 
We implement our test statistic \eqref{eq:test_stat} for the largest tree, using 10 equidistant points of $\Dd = [0.0326;0.0407]$ -- where 80\% of the growth rates lie -- using the covariance matrix estimator \eqref{eq:varest}. For the largest tree of 655 cells, our test strongly reject the null hypothesis (p-value $< 10^{-3}$). In the same way, the null hypothesis is strongly rejected for the 10 largest trees available (from 443 cells to 655). Thus, we may conclude to asymmetry in the transmission of the growth rate from one cell to its offspring.
Admittedly, our test does not take into account the influence of missing data and the level of our test for small trees is poor (recall Table \ref{tab:test_BAR}, $(\text{{\bf Case}} \boldsymbol{=})$, Column $n=8$). Thus one should take this conclusion with caution.
%
%


\begin{figure}[h!]
\centering
\subfigure[$N = 655$]{\includegraphics[width=7cm]{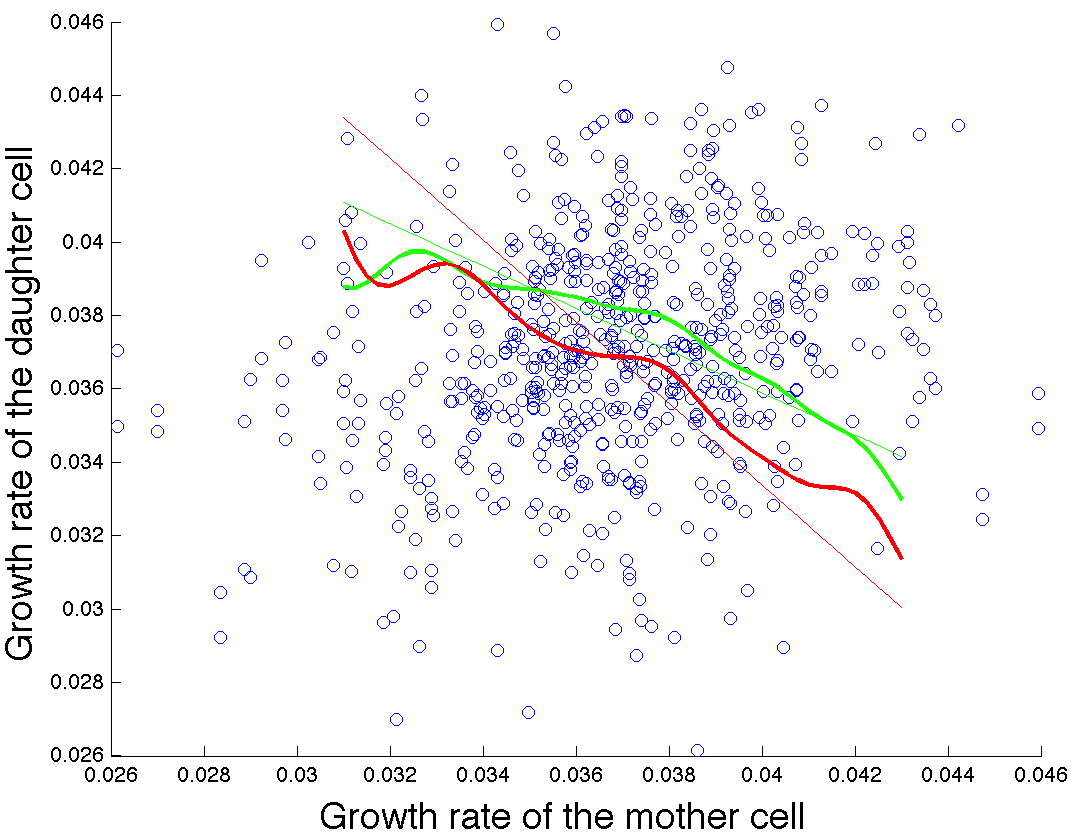} \label{fig:Stew655}}
\subfigure[$N = 446$]{\includegraphics[width=7cm]{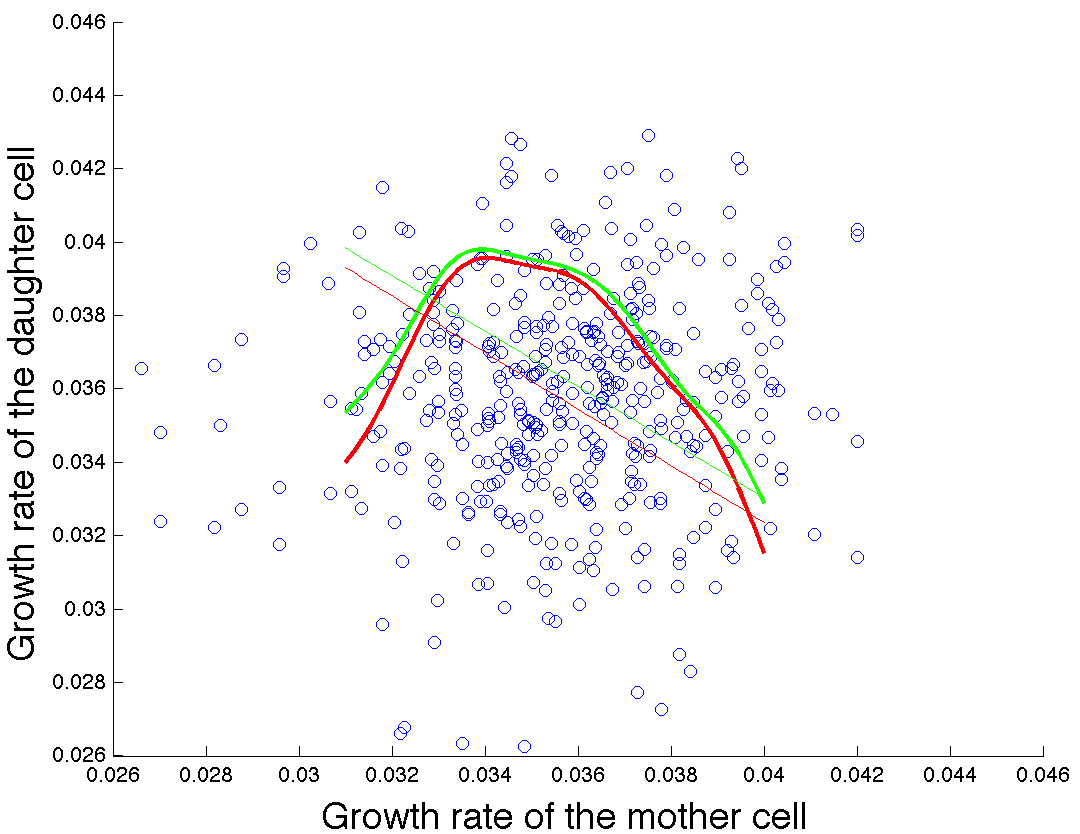} \label{fig:Stew446}}
\caption{[Real data from Stewart {\it et al.} \cite{SMPT05}] {\it  Points  $(X_u,X_{u0})$ and $(X_u,X_{u1})$ for $N$ cells $u\in \TT_9$. Bold green curve (resp. Thin green line): reconstruction of $x\leadsto f_0(x)$ with $x\leadsto \bar f_{0,n}(x)$ (resp. with a linear estimator). Bold red curve (resp.  Thin red line): reconstruction of $x\leadsto f_1(x)$ with $x\leadsto \bar f_{1,n}(x)$  (resp. with a linear estimator).}}
\end{figure}
%


\section{\textsc{Discussion}} \label{sec:discussion}
\subsubsection*{Recursive estimators}
We could estimate $\big( f_0(x), f_1(x)\big)$ the autoregressive functions at point $x\in\Dd$ from the observations $(X_u)_{ u\in\TT_{n+1}}$ by
\begin{equation*}
\Big( \widehat{f}_{\iota, n}(x) =
\frac{\sum\limits_{m=0}^{n}\sum\limits_{u \in \GG_{m}}  K_{h_m} (x-X_u) X_{u\iota}}{
\sum\limits_{m=0}^{n}\sum\limits_{u \in \GG_{m}} K_{h_m}(x-X_u)} \, ,  \iota \in \{0,1\} \Big)
\end{equation*}
with the collection of bandwidths $(h_{m} = |\GG_{m}|^{-\alpha})_{0 \leq m \leq n}$ for $\alpha\in(0,1)$. 
These estimators can be seen as a version of recursive Nadaraya-Watson estimators when the index set has a binary tree structure. We stress that our results also hold  for this alternative procedure.

\subsubsection*{Heteroscedasticity} Given two functions $\sigma_0,\sigma_1 : \statespace \rightarrow [0,\infty)$, we could consider the generalized autoregressive equations
$$
X_{u0} = f_0(X_u) + \sigma_0(X_u) \varepsilon_{u0} \quad \text{ and } \quad X_{u1} = f_1(X_u) + \sigma_1(X_u)  \varepsilon_{u1}
$$
with $\EE[\varepsilon_{u0}^2] = \EE[\varepsilon_{u1}^2] = 1$ and $\EE[\varepsilon_{u0} \varepsilon_{u1}] = \varrho$ where $\varrho\in(-1,1)$.
Assuming $0 < \inf_{x \in\statespace} \sigma_\iota(x)\leq \sup_{x\in\statespace} \sigma_\iota(x) < \infty$ for $\iota \in \{ 0,1\}$, Theorems~\ref{thm:asymptotic_normality} and~\ref{thm:asymptotic_normality2} still hold with 
$$
\boldsymbol{\Sigma}_2(x) = | K |_2^2 \big(\nu(x)\big)^{-1} \begin{pmatrix} \sigma_0^2(x) & \varrho \sigma_0(x) \sigma_1(x) \\ \varrho \sigma_0(x) \sigma_1(x) & \sigma_1^2(x) \end{pmatrix}.
$$
The estimation of the variance functions $\sigma_0$ and $\sigma_1$ would be interesting but the theoretical study of such estimators lies here beyond the scope of this work.


\subsubsection*{Uniform test} The asymmetry test we have built is based on the choice of a grid of points on $\statespace$. A theoretical result is needed in order to build a uniform test on a interval $\Dd\subset \statespace$. More precisely, to achieve such a uniform test we should study the asymptotic behaviour of
$$
\sup_{x\in\Dd} \big| \widehat{f}_{\iota,n}(x)-f_{\iota}(x) \big| , \quad \iota\in\{0,1\}.
$$
This asymptotic study lies in the scope of the theory of extrema. One can see the study of Liu and Wu \cite{LiuWu} for autoregressive processes of order 1: an asymptotic Gumbel behaviour is highlighted for the Nadaraya-Watson type estimator of the autoregressive function. 
Alternatively, studying the limit distribution of
$$
\int_\Dd \big( \widehat f_{\iota, n}(x) - f_{\iota}(x) \big)^2 dx  , \quad \iota\in\{0,1\},
$$
we could derive an other criterion to discriminate between $f_0 = f_1$ and $f_0 \neq f_1$.


\subsubsection*{Moderate deviations principle} 
The work of Bitseki Penda {\it et al.} \cite{BHO} brings deviations inequalities  which  enable us to derive a moderate deviations principle for the estimators $\big(\widehat f_{0,n}(x),\widehat f_{1,n}(x)\big)$.
The results of \cite{BHO} are valid under a uniform ergodicity assumption for the tagged-branch chain, which can be achieved restricting ourselves to the class $\mathcal F(\gamma=0,\ell)$.


\section{\textsc{Proofs}} \label{sec:proofs}
The notation $\lesssim$ means up to some constant independent of $n$ and uniform on the class $(f_0,f_1)\in \Ff(\gamma,\ell)^2$.

For a $\mathfrak B$-measurable function $g : \statespace \rightarrow \RR$ and a measure $\mu$ on $(\statespace,\mathfrak B)$ we define  $\mu(g) = \int _{\statespace} g(x) \mu(dx)$. 
For $K \subseteq \statespace$ let
$$
| g |_1 = \int_\statespace |g(y)| dy \, , \quad  | g |_2^2 = \int_\statespace g(y)^2 dy \, , \quad |g|_K = \sup_{y \in K} |g(y)|
$$
and $|g|_\infty = |g|_\statespace$. For a function $g : \statespace^2 \rightarrow \RR$ and $K,K' \subseteq \statespace$ let
$$
|g|_{K,K'} = \sup_{(x,x') \in K\times K'} |g(x,x')|.
$$

The following lemma is well-known  in the general setting of bifurcating Markov chains including our NBAR model (see Delmas and Marsalle \cite{DM10}, Lemma 2.1 and Guyon \cite{Guyon}, Equation (7)) and highlights the key role of the tagged-branch Markov chain. We prove it in Appendix for the sake of completeness. Introduce 
\begin{equation} \label{eq:TtransitionBAR}
\Ttransition(x,dydz) =  \noisedensity(y-f_0(x),z-f_1(x)) dy dz,
\end{equation}
Markov kernel from $(\mathbb R,\mathfrak B)$ to $(\mathbb R\times \mathbb R,\mathfrak B \otimes \mathfrak B)$.
\begin{lem}[Many-to-one formulae] \label{lem:Mto1} 
Let $(X_u)_{u\in\TT}$ be a NBAR process, with any initial probability measure $\mu(dx)$ on $(\statespace,\mathfrak B)$ for $X_\emptyset$ such that $\mu\big((1+|\cdot|)^2\big)<\infty$. Then for $g : \statespace \rightarrow \RR$ such that $|g(x)|\leq1+|x|$ for any $x\in \statespace$, we have
\begin{equation} \label{eq:Mto11}
\EE_\mu \Big[ \sum_{u \in\GG_m} g(X_u) \Big]   = |\GG_m|\EE_\mu \big[g(Y_m) \big]  = |\GG_m| \mu \big( \Qq^m g \big)
\end{equation}
with $\Qq$ defined by \eqref{eq:QtransBAR} and
\begin{equation} \label{eq:Mto12}
\EE_\mu \Big[\sum_{(u,v) \in\GG_m^2 \atop u\neq v} g(X_u) g(X_v) \Big]  =  |\GG_m| \sum_{l = 1}^m 2^{l -1} \mu \Big( \Qq^{m-l} \big( \Ttransition ( \Qq^{l - 1} g \otimes \Qq^{l  - 1} g ) \big) \Big),
\end{equation}
with $\Ttransition$ defined by \eqref{eq:TtransitionBAR}.
\end{lem}

\subsection{Preliminary}

Set $\VV(x) = 1+|x|$ for $x\in\statespace$. It plays the role of the Lyapunov function in the following
\begin{lem}[Ergodicity] \label{lem:ergodicity} 
Let $\gamma \in (0,1/2)$ and $\ell>0$, let $r>0$ and $\lambda>2$. For every $(f_0,f_1)\in\Ff(\gamma,\ell)^2$, for every $\noisedensity$ such that $(G_0,G_1)\in \Gg(r,\lambda)^2$ satisfy Assumption~\ref{ass:ergodicity2}, 
the Markov kernel $\Qq$ admits a unique invariant probability measure $\nu$ of the form $\nu(dx) = \nu(x) dx$ on $(\statespace,\mathfrak B)$.
Moreover, for every $\noisedensity$ such that $(G_0,G_1)\in \Gg(r,\lambda)^2$ satisfy Assumption~\ref{ass:ergodicity2}, there exist a constant $R >0$ and $\rho \in(0,1/2)$ such that
\begin{equation*}\label{eq:lyapunov}
\sup_{(f_0,f_1)} \sup_{|g|\leq \VV}|\Qq^{m}g(x) - \nu(g)|\leq R \, \VV (x) \, \rho^{m}, \quad x \in \statespace, \quad m \geq 0,
\end{equation*}
where the supremum is taken among all functions $(f_0,f_1)\in \Ff(\gamma,\ell)^2$ and among all functions $g:\statespace \rightarrow \RR$ which satisfy $|g(x)|\leq \VV(x)$ for all $x\in \statespace$. 
\end{lem}

\begin{proof}[Proof of Lemma~\ref{lem:ergodicity}]
We shall rely on the results of Hairer and Mattingly \cite{HairerMattingly}.
\vip
\noindent {\it Step 1.}
In order to make use of Theorem 1.2 of  \cite{HairerMattingly} we shall verify their Assumptions 1 and 2.
%
Since $Y_1 = f_{\iota_1}(Y_0) + \ep'_1$ where $\iota_1$ is drawn according to the Bernoulli distribution with parameter $1/2$ and $\ep'_1$ has density $(G_0+G_1)/2$, we get
\begin{align*}\label{eq:Hairer-Mattingly1}
\Qq (|\cdot|)(x) & = \EE_x \big[|Y_{1}|\big] \leq \EE\big[|f_{\iota_1}(x)| \big] + \EE[|\ep'_1|] \leq \gamma |x| + M_0
\end{align*}
using $(f_0,f_1) \in \Ff(\gamma,\ell)^2$, with $M_0 = \ell + \EE[|\ep'_1|]$ as defined previously. We have $\gamma\in(0,1)$ and $M_0 \geq0$, so that is Assumption 1 in \cite{HairerMattingly} (with their $V(y) = |y|$). \\

Set $\Cc = \{x\in\statespace; |x|\leq M_1 \}$ where $M_1$ comes from Assumption~\ref{ass:ergodicity2}. For any $A\in\mathfrak B$ and $x\in\Cc$, using the expression of $\Qq$ given by \eqref{eq:QtransBAR},
\begin{align*}
\Qq(x,A) 
& \geq \frac{1}{2} \int_{A\cap \Cc} G_0\big(y-f_0(x)\big) dy + \frac{1}{2} \int_{A\cap \Cc} G_1\big(y-f_1(x)\big) dy. 
\end{align*}
For $(x,y)\in\Cc^2$, we have $|y - f_\iota(x)| \leq (1+\gamma)M_1+\ell$ for $\iota\in\{0,1\}$. Thus
\begin{equation*}\label{eq:Hairer-Mattingly1}
\inf_{x\in \Cc} \Qq(x,A) \geq 2 M_1 \delta\big( (1+\gamma)M_1+\ell \big) \frac{|A\cap\Cc|}{|\Cc|} \quad \forall A\in \mathfrak B,
\end{equation*}
where $\delta(\cdot)$ is defined by \eqref{eq:minorg} and $|A|$ denotes the Lebesgue measure of $A\in\mathfrak B$. That is Assumption~2 in \cite{HairerMattingly} with $\alpha = 2 M_1 \delta\big( (1+\gamma)M_1+\ell \big)>0$.  The existence and uniqueness of an invariant probability measure $\nu$ follows from Theorem 1.2 of \cite{HairerMattingly}. 
Moreover $\nu$ is absolutely continuous with respect to the Lebesgue measure, since $\Qq(x,dy)$ defined by \eqref{eq:QtransBAR} itself is absolutely continuous with respect to the Lebesgue measure. 
By Assumption~\ref{ass:ergodicity2}, for $M_1$ satisfying \eqref{eq:ergodicity2}, there exists some $\a_0\in(0,1/2)$ such that 
\begin{equation} \label{eq:ergodicity22}
\alpha = 2M_1 \delta\big( (1+\gamma)M_1+\ell \big) > 1/2 + \alpha_0.
\end{equation}
We set $\beta = \alpha_{0}/M_0$. For all $x\in \statespace$ we pick $\Qq^{m}\delta_{x}$ and $\nu$ for $\mu_1$ and $\mu_2$ in Theorem 1.3 of \cite{HairerMattingly} and apply it recursively. We conclude that for any function $g$ such that $|g(x)| \leq (1+\beta|x|)$ for all $x\in\statespace$, for some positive constant $C$, we have
\begin{equation*}
|\Qq^{m}g(x) - \nu(g)| \leq C\rho^m (1+\beta|x|)
\end{equation*}
with $C= 1+\int_\statespace (1+\b|x|)\nu(x)dx<\infty$.
\vip
\noindent {\it Step 2.} A precise control of $\rho$ with respect to $\gamma$, $M_0$ and $\alpha$ is established in \cite{HairerMattingly}. Set $\gamma_0 = \gamma+2M_0/M_1 +\eta \in (\gamma+2M_0/M_1,1)$ where $\eta$ comes from Assumption~\ref{ass:ergodicity2}. Theorem 1.3 of \cite{HairerMattingly} states one can take 
$$
\rho = \big(1 - (\alpha - \alpha_0)\big) \vee \big( \frac{2+M_1\beta\gamma_0}{2+M_1\beta} \big).
$$
Condition \eqref{eq:ergodicity22} gives immediately
$1-(\alpha-\alpha_0) < 1/2.$ Note that
$$
\frac{2+M_1\beta\gamma_0}{2+M_1\beta} < \frac{1}{2} \quad \Longleftrightarrow \quad M_1 > \frac{(1+2\alpha_0)M_0}{1/2-\gamma-\eta}
$$
which is satisfied for $\alpha_0\in(0,1/2)$ with our choice of $M_1$ (which satisfies \eqref{eq:ergodicity2}).
Thus Assumption~\ref{ass:ergodicity2} guarantees  we can take $\rho < 1/2$. 
\vip
It immediately follows from Step 1 and Step 2 that, for any function $g$ such that $|g| \leq \VV$, for some positive constant $R$, we have
\begin{equation*}
|\Qq^{m}g(x) - \nu(g)| \leq R\rho^{m}\VV(x)
\end{equation*}
with $\VV(x) = 1+|x|$, as asserted. This bound holds uniformly over $(f_0,f_1)\in \Ff(\gamma,\ell)^2$ by construction (the uniform choice of $\rho$ is guaranteed by Step 2 and for a uniform choice of $C$ in Step 1 recall that $(G_0,G_1)\in\Gg(r,\lambda)^2$).
\end{proof}
Step 2 of this proof highlights that Assumption~\ref{ass:ergodicity2} is written to readily obtain $\rho<1/2$. Note that to prove the existence of some $\rho\in (0,1)$ in Step 1, one only need the existence of some $M_1> 2M_0/(1-\gamma)$ such that $2 M_1 \delta\big( (1+\gamma)M_1+\ell \big)>0$ with $\delta(\cdot)$ defined by \eqref{eq:minorg}.
\subsection{Estimation of the density of the invariant measure}
For $x\in \Dd$, set
\begin{equation} \label{nuhat}
\widehat{\nu}_n(x) = \frac{1}{|\TT_n|}\sum\limits_{u \in\TT_n} K_{h_n} (x-X_u),
\end{equation}
a kernel estimator of the density $\nu$ of the invariant measure of the tagged-branch chain $Y$ of transition~$\Qq$.
\begin{prop} \label{prop:nuhat_L2}
%
Let $\gamma\in(0,1/2)$ and $\ell>0$, let $r>0$ and $\lambda>3$.
Specify $\widehat{\nu}_n$ with a kernel $K$ satisfying Assumption~\ref{ass:kernel} for some $n_0>0$ and
$$
h_n \propto |\TT_n|^{-1/(2\beta+1)}.
$$
For every $L'>0$ and $0<\beta<n_0$, for every $\noisedensity$ such that $(G_0,G_1)\in \big( \Gg(r,\lambda)\cap \Hh^{\beta}_\statespace(L')\big)^2$ satisfy Assumptions~\ref{ass:ergodicity2} and~\ref{ass:numinor}, for every compact interval $ \Dd\subset \statespace$ with nonempty interior and every $x$ in the interior of $\Dd$,
$$
\sup_{(f_0,f_1)} \EE_{\mu} \Big[ \big(\widehat{\nu}_n(x) - \nu(x) \big)^2 \Big] \lesssim |\TT_n|^{\frac{-2\beta}{2\beta+1}}
$$
where the supremum is taken among all  functions $(f_0,f_1) \in \Ff(\gamma,\ell)^2$, for any initial probability measure $\mu(dx)$ on $\statespace$ for $X_\emptyset$ such that $\mu\big((1+|\cdot|)^2\big)<\infty$.
\end{prop}

\begin{proof}[Proof] 
The usual bias-variance decomposition can be written here as
\begin{align*}
\EE\Big[ \big(\widehat{\nu}_n(x) - \nu(x) \big)^2 \Big] &= \EE\Big[ \big(\frac{1}{|\TT_n|}\sum\limits_{u \in\TT_n} K_{h_n} (x-X_u) - \nu(x) \big)^2 \Big]  \\
& = \EE\Big[ \big(\frac{1}{|\TT_n|}\sum\limits_{u \in\TT_n} K_{h_n} (x-X_u) - K_{h_n}\star\nu(x) \big)^2 \Big]  + \big( K_{h_n}\star\nu(x) - \nu(x) \big)^2
\end{align*}
where $\star$ stands for the convolution.
For $(G_0,G_1) \in \Hh^\b_\statespace(L')$, we have $\nu\in\Hh^\b_\Dd(L'')$ for some $L''>0$: since $\nu$ is invariant for $\Qq$, using \eqref{eq:QtransBAR} which defines $\Qq$, 
we can write
$$
\nu(y) = \int_\statespace \nu(x) \Qq(x,y) dx = \frac{1}{2} \int_\statespace \nu(x) \Big(G_0\big( y - f_0(x) \big) + G_1\big( y - f_1(x) \big)\Big) dx,
$$
where we immediately see that the regularity of $\nu$ is inherited from the regularity of $G_0$ and $G_1$ the marginals of the noise density $\noisedensity$.
By a Taylor expansion up to order $\lfloor \beta \rfloor$ (recall that the number $n_0$ of vanishing moments of $K$ in Assumption~\ref{ass:kernel} satisfies $n_0>\beta$), we obtain
\begin{equation} \label{nuhat_biais}
\big( K_{h_n}\star\nu(x) - \nu(x) \big)^2 \lesssim h_n^{2\beta},
\end{equation}
see for instance Proposition 1.2 in Tsybakov \cite{Tsybakov}.
In addition, we claim that
\begin{equation} \label{nuhat_var}
\EE\Big[ \big(\frac{1}{|\TT_n|}\sum\limits_{u \in\TT_n} K_{h_n} (x-X_u) - K_{h_n}\star\nu(x) \big)^2 \Big]  \lesssim (|\TT_n| h_n)^{-1}.
\end{equation}
Choosing $h_n \propto |\TT_n|^{-1/(2\beta+1)}$ brings the announced result. Let us now prove \eqref{nuhat_var} in two steps.
\vip
\noindent {\it Step 1. Result over one generation}. 
We heavily rely on the following controls:
\begin{lem} \label{lem:gagnerh} Let $F$ be a bounded function with compact support and $G$ be a locally bounded function. For $h_n>0$ and $x$ in the interior of $\Dd$,
we define the function $H_{n} : \statespace \rightarrow \RR$ by
 $$H_n(\cdot) = F\big(h_n^{-1}(x-\cdot)\big)G(\cdot).$$
For $h_n$ such that $x- h_n {\rm supp}(F) \subset \Dd$, we have
\begin{itemize}
\item[(i)] $|\Qq H_n|_{\infty} \leq h_n |F|_1 |G|_\Dd  |\Qq|_{S,\Dd}$ and $|\nu(H_n)| \leq h_n |F|_1 |G|_\Dd |\nu|_\Dd$.
\item[(ii)] Under Assumption~\ref{ass:ergodicity2}, for $m\geq 1$, $y\in\statespace$,
$$|\Qq^m H_n(y) - \nu(H_n) | \lesssim h_n \wedge \big(\VV(y) \rho^m\big)
$$
up to the constant $\max\{ |F|_1 |G|_\Dd \big(|\Qq|_{\statespace,\Dd}+|\nu|_\Dd \big), R |F|_\infty |G|_{\Dd} \}$.
\end{itemize}
\end{lem}
\noindent The proof of this lemma is postponed to the Appendix. 
Note that we have $|\nu|_\Dd\leq|\Qq|_{\statespace,\Dd} \leq |\Qq|_{\statespace,\statespace} \leq (|G_0|_\infty+|G_1|_\infty)/2<\infty$ (recall that $\nu$ is invariant for $\Qq$ defined by \eqref{eq:QtransBAR} and $(G_0,G_1)\in\Gg(r,\lambda)^2$).
Set
$$
H_n (\cdot ) = K\big(h_n^{-1} (x-\cdot) \big)  \quad \text{and} \quad \widetilde{H}_n(\cdot) = H_n(\cdot) - \nu (H_n),
$$
with $h_n$ sufficiently small such that $x-h_n\text{supp}(K) \subset \Dd$.
Pick $m\geq 1$.
On the one hand,
\begin{align*}
\EE_\mu \Big[ \sum_{u \in\GG_m} \widetilde{H}_n^2(X_u) \Big]   & = |\GG_m| \mu \big( \Qq^m \widetilde{H}_n^2 \big) \lesssim |\GG_m| h_n
\end{align*}
relying on the many-to-one formula \eqref{eq:Mto11} and Lemma~\ref{lem:gagnerh}(i).
Inspired by Doumic, Hoffmann, Krell and Robert \cite{DHKR1} (proof of Proposition~8), set $l^{*} = \lfloor |\log h_n| / |\log \rho |\rfloor$. Since $\mu \big( \Qq^{m-l} \big( \Ttransition ( \VV \otimes \VV ) \big) \big) < \infty$ (use $\VV(x) = 1+|x|$ for $x\in \statespace$, $(f_0,f_1) \in \Ff(\gamma,\ell)^2$ and finally $\mu(\VV^2)<\infty$, one can look at Lemmae~25 and 26 of Guyon \cite{Guyon}), by the many-to-one formula \eqref{eq:Mto12},
$$
\EE_\mu \Big[ \sum_{u \neq v \in \GG_m} \widetilde{H}_n(X_u) \widetilde{H}_n(X_v) \Big] \lesssim |\GG_m| \Big( \sum_{l = 1}^{l^*} 2^{l-1} h_n^2  + \sum_{l =
l^* +1}^{m} 2^{l-1} \rho^{2( l-1)} \Big) \lesssim 2^m h_n \, ,
 $$
using the first upper-bound given by Lemma~\ref{lem:gagnerh}(ii) for before $l^*$, the second one after $l^*$, and using $\rho \in(0, 1/2)$.
Finally, we conclude that
$$
\EE_\mu \Big[\big(\sum\limits_{u \in \GG_m} \widetilde{H}_n(X_u)\big)^{2}\Big] \lesssim |\GG_m|h_n.
$$
Thus, we have uniformly over $(f_0,f_1)\in\Ff(\gamma,\ell)^2$,
\begin{equation} \label{eq:1ge}
\EE_\mu\Big[ \big(\frac{1}{|\GG_m|}\sum\limits_{u \in\GG_m} K_{h_n} (x-X_u) - K_{h_n}\star\nu(x) \big)^2 \Big] \lesssim (|\GG_m| h_n)^{-1}, \quad m\geq 1.
\end{equation}
\vip
\noindent {\it Step 2. Result over a subtree}. We rely on the previous inequality \eqref{eq:1ge}.
Decomposing by generation and by the triangle inequality, we obtain
\begin{align*}
\EE_\mu \Big[  \big(\frac{1}{|\TT_n|} & \sum\limits_{u \in\TT_n} K_{h_n} (x-X_u) - K_{h_n}\star\nu(x) \big)^2 \Big]  \\
& \leq \bigg(\sum_{m=0}^{n} \frac{|\GG_{m}|}{|\TT_n|} \Big(\EE_\mu \Big[\big(\frac{1}{|\GG_{m}} \sum_{u \in\GG_{m}} K_{h_n} (x-X_u) - K_{h_n}\star\nu(x) \big)^2 \Big]\Big)^{1/2}\bigg)^2 \\
& \lesssim \Big(  \frac{|\GG_{0}|}{|\TT_n|} h_n^{-1} + \sum_{m=1}^{n} \frac{|\GG_{m}|}{|\TT_n|}(|\GG_{m}|h_n)^{-1/2}\Big)^{2} \lesssim (|\TT_n|h_n)^{-1}.
\end{align*}
%
This proves inequality \eqref{nuhat_var} we claimed and the proof is now complete. Note that we have removed the log-term which appears in Proposition 8 of Doumic {\it et al.} \cite{DHKR1}.
\end{proof}

\begin{prop} \label{prop:nuhat_ps}
In the same setting as in Proposition~\ref{prop:nuhat_L2},
$$
\widehat{\nu}_n(x) \rightarrow \nu(x) , \quad \PP_\mu - a.s.
$$
as $n\rightarrow \infty$. 
\end{prop}
\begin{proof}[Proof]
Write
\begin{equation*}
\widehat{\nu}_n(x) - \nu(x) = \big(\frac{1}{|\TT_n|}\sum\limits_{u \in\TT_n} K_{h_n} (x-X_u) - K_{h_n}\star\nu(x) \big) + \big( K_{h_n}\star\nu(x) - \nu(x) \big).
\end{equation*}
From \eqref{nuhat_biais}, we deduce that
\begin{equation}\label{eq:bochner}
| K_{h_n}\star\nu(x) - \nu(x) | \rightarrow 0  \, \text{ as } \, n\rightarrow \infty.
\end{equation}
Note that we could obtain \eqref{eq:bochner} invoking the result stated by Theorem 2.1.1 of Prakasa Rao \cite{PrakasaRao83}, result also known as the Bochner Lemma (see section 7.1.2 of Duflo \cite{Duflo98}).
Using \eqref{nuhat_var} for $h_n \propto |\TT_n|^{-\alpha}$ with $\alpha \in (0,1)$,
$$
\sum_{n\geq 0} \EE_\mu \Big[ \big(\frac{1}{|\TT_n|}\sum\limits_{u \in\TT_n} K_{h_n} (x-X_u) - K_{h_n}\star\nu(x) \big)^2 \Big] < \infty,
$$
and by the Borel-Cantelli lemma, we deduce that
$$
\Big| \frac{1}{|\TT_n|}\sum_{u \in\TT_n} K_{h_n} (x-X_u) - K_{h_n}\star\nu(x) \Big| \rightarrow 0  , \quad \PP_\mu - a.s.
$$
as $n\rightarrow \infty$. Thus $| \widehat{\nu}_n(x) - \nu(x) | \rightarrow 0$, $\PP_\mu - a.s.$ as $n\rightarrow \infty$.
\end{proof}


\subsection{Proof of Theorem~\ref{thm:upper_rateBAR}}
For $x$ in the interior of $\Dd$, for $\iota \in \{ 0,1\}$, we plan to use the decomposition
\begin{align*}
 \widehat{f}_{\iota,n}(x)  - f_\iota(x) & = \frac{M_{\iota ,n}(x)  + L_{\iota ,n}(x) }{\widehat{\nu}_n(x)\vee \varpi_n} - \frac{(\nu f_\iota) (x)}{\nu(x)} \\
& = \frac{M_{\iota ,n}(x)}{\widehat{\nu}_n(x)\vee \varpi_n} + \frac{L_{\iota ,n}(x) - (\nu f_\iota) (x) }{\widehat{\nu}_n(x)\vee \varpi_n} - \frac{\widehat{\nu}_n(x)\vee \varpi_n - \nu(x)}{\widehat{\nu}_n(x)\vee \varpi_n} f_\iota(x)
\end{align*}
with 
\begin{align}
M_{\iota ,n}(x) & = \frac{1}{|\TT_n|} \sum\limits_{u \in \TT_n} K_{h_n} (x-X_u) \ep_{u\iota}, \label{eq:partmart} \\
L_{\iota,n}(x) & = \frac{1}{|\TT_n|} \sum_{u \in \TT_n}  K_{h_n}(x-X_u) f_{\iota}(X_u) \label{eq:partnuf}.
\end{align}
Thus
\begin{multline} \label{3terms}
\EE_\mu \Big[ \big(  \widehat{f}_{\iota,n}(x)  - f_\iota(x)  \big)^2 \Big]  \lesssim \varpi_n^{-2} \Big(\EE_\mu \Big[  \big( M_{\iota ,n}(x) \big)^2 \Big]  + \EE_\mu \Big[ \big( L_{\iota ,n}(x) - (\nu f_\iota) (x) \big)^2 \Big]  \\ + \EE_\mu \Big[ \big( \widehat{\nu}_n(x)\vee \varpi_n - \nu(x) \big)^2 \Big] \Big) 
\end{multline}
using $| f_\iota|_\Dd < \infty$, uniformly over the class $\Ff(\gamma,\ell)$ for $\Dd$ compact interval.
We successively treat the three terms in Steps from 1 to 3.
\vip
\noindent \textit{Step 1. Term $M_{\iota ,n}(x)$.} For all $m \geq 1$ and $\iota\in\{0,1\}$ fixed, the sequence $( \ep_{u\iota})_{ u \in \GG_m}$ is a family independent random variables such that $\EE[\varepsilon_{u\iota}^2] = \sigma_\iota^2$. Thus
\begin{align*}
\EE_\mu \Big[\big(\frac{1}{|\GG_{m}|}\sum_{u \in\GG_m} K_{h_n} (x-X_u )\ep_{u\iota} \big)^{2} \Big] & = 
\frac{\sigma_\iota^{2}}{|\GG_{m}| h_n^{2}}  \EE_\mu \Big[  \frac{1}{|\GG_m|} \sum_{u \in\GG_{m}} K^{2} \big(h_n^{-1}(x-X_u)\big)\Big] \\
& = \frac{\sigma_\iota^{2}}{|\GG_{m}|h_n^{2}} \mu\Big(\Qq^m K^{2}\big(h_n^{-1}(x-\cdot)\big)\Big)  \lesssim (|\GG_m| h_n)^{-1}
\end{align*}
by the many-to-one formula \eqref{eq:Mto11} and using Lemma~\ref{lem:gagnerh}(i).
The result over a subtree follows by the triangle inequality (as in Step 2 of the proof of Proposition~\ref{prop:nuhat_L2}),
\begin{equation} \label{termM}
\EE_\mu \Big[  \big( M_{\iota ,n}(x) \big)^2 \Big] \lesssim (|\TT_n| h_n)^{-1}.
\end{equation}
\vip
\noindent \textit{Step 2. Term $L_{\iota ,n}(x)$.} By usual the bias-variance decomposition,
\begin{multline*}
\EE_\mu \Big[ \big( L_{\iota ,n}(x) - (\nu f_\iota) (x) \big)^2 \Big]  = \EE_\mu \Big[ ( \frac{1}{|\TT_n|} \sum_{u \in \TT_n}  K_{h_n}(x-X_u) f_{\iota}(X_u) - K_{h_n}\star (\nu f_\iota)(x) \big)^2  \Big] 
\\+\big(  K_{h_n}\star (\nu f_\iota)(x) - (\nu f_\iota)(x)\big)^2.
\end{multline*}
First, since $(\nu f_\iota) \in \Hh^\beta_\Dd(L'')$ for some constant $L''>0$ and since Assumption~\ref{ass:kernel} is valid with $n_0>\beta$,
$$
\big(  K_{h_n}\star (\nu f_\iota)(x) - (\nu f_\iota)(x)\big)^2 \lesssim h_n^{2\beta}.
$$
Secondly, we do the same study as in the proof of Proposition~\ref{prop:nuhat_L2} for
$$
H_n(\cdot) = K\big( h_n^{-1} \big(x-\cdot)\big) f_\iota(\cdot),
$$
relying on Lemma~\ref{lem:gagnerh} (using $|f_\iota|_{\Dd}<\infty$ uniformly over the class $\Ff(\gamma,\ell)$, with $\Dd$ compact interval), with $h_n$ sufficiently small such that $x-h_n\text{supp}(K) \subset \Dd$.
We obtain
$$
\EE_\mu \Big[ ( \frac{1}{|\TT_n|} \sum_{u \in \TT_n}  K_{h_n}(x-X_u) f_{\iota}(X_u) - K_{h_n}\star (\nu f_\iota)(x) \big)^2  \Big]  \lesssim (|\TT_n|h_n)^{-1}.
$$
Thus
\begin{equation} \label{termL}
\EE_\mu \Big[ \big( L_{\iota ,n}(x) - (\nu f_\iota) (x) \big)^2 \Big] \lesssim h_n^{2\beta} + (|\TT_n|h_n)^{-1}.
\end{equation}
\vip
\noindent \textit{Step 3. Denominator  $\widehat{\nu}_n(x)\vee \varpi_n$.} We prove the following lemma in Appendix.
\begin{lem} \label{lem:minorationnu}
Let $\gamma\in(0,1/2)$ and $\ell>0$, let $r>0$ and $\lambda>3$.
For every $\noisedensity$ such that $(G_0,G_1)\in \Gg(r,\lambda)^2$ satisfy Assumptions~\ref{ass:ergodicity2} and~\ref{ass:numinor}, there exists $d = d(\gamma,\ell , G_0,G_1)>0$ such that for every $\Dd \subset [-d,d]$,
$$\inf_{(f_0,f_1)} \inf_{x\in \Dd} \nu(x) > 0$$ 
where the infimum is taken among all functions $(f_0,f_1)\in\Ff(\gamma,\ell)^2$.
\end{lem}
Relying first on Lemma~\ref{lem:minorationnu}, we choose $n$ large enough such that 
\begin{equation} \label{eq:varpi}
0 < \varpi_n \leq \frac{1}{2} \inf_{(f_0,f_1)} \inf_{x\in\Dd} \nu(x),
\end{equation}
 and we get
\begin{equation} \label{termDenominateur}
\EE_\mu \Big[ \big( \widehat{\nu}_n(x)\vee \varpi_n - \nu(x) \big)^2 \Big]  \lesssim \EE_\mu \Big[ \big( \widehat{\nu}_n(x)- \nu(x) \big)^2 \Big] \lesssim h_n^{2\beta} + (|\TT_n|h_n)^{-1},
\end{equation}
uniformly over $(f_0,f_1)\in\Ff(\gamma,\ell)^2$, using the upper-bound obtained in the proof of Proposition~\ref{prop:nuhat_L2} for the second inequality.
\vip
\noindent Finally, gathering \eqref{termM}, \eqref{termL} and \eqref{termDenominateur} in \eqref{3terms} and choosing $h_n \propto |\TT_n|^{-1/(2\beta+1)}$ we obtain the asserted result.

\begin{rk}
The threshold $\varpi_n$ should be chosen such that it inflates the upper-rate of convergence of a slow factor only. Typically, $\varpi_n = (\ln n)^{-1}$ is suitable. Looking carefully at the proof, see \eqref{eq:varpi}, we actually see that $\varpi_n \rightarrow 0$ as $n\rightarrow\infty$ is not necessary. One could choose, $\varpi_n = \varpi$ with
$$
\varpi = \frac{1}{2} \inf_{(f_0,f_1)} \inf_{x\in \Dd} \nu(x) > 0
$$ 
where the infimum is taken among all $(f_0,f_1)\in\Ff(\gamma,\ell)^2$ and where $\varpi>0$ is guaranteed by Lemma~\ref{lem:minorationnu}. However, to calibrate in practice the threshold in such a way is not possible since we cannot compute~$\varpi$. \\ 

\end{rk}

\subsection{Proof of  Theorem~\ref{thm:lower_rateBAR}}
In the following, $\PP^n_{(f_0,f_1)}$ will denote the law on $\statespace^{|\TT_{n+1}|}$ of the vector $(X_u, u\in \TT_{n+1})$, NBAR process, in the sense of Definition~\ref{def:NBAR}, driven by the autoregressive functions $f_0$ and $f_1$ with initial probability measure $\mu(dx)$ on $\statespace$ for $X_\emptyset$ and with a Gaussian noise
\textit{i.e.} 
$$
\noisedensity(x,y) = \big( 2\pi (\sigma_0^2 \sigma_1^2(1 -\rho^2) \big)^{-1/2} \exp\Big(-\frac{\sigma_1^2 x^2 - 2\sigma_0 \sigma_1 \rho xy + \sigma_0^2 y^2}{2\big(\sigma_0^2 \sigma_1^2(1 -\rho^2)\big)} \Big), \quad (x,y)\in \statespace^2,
$$
with $\sigma_0,\sigma_1>0$ and $\rho\in(-1,1)$.
 When $f_0=f_1=f$, we shorten $\PP^n_{(f_0,f_1)}$ into $\PP^n_f$. We denote by $\EE^n_f[\cdot]$ the expectation with respect to $\PP^n_f$.
\vip
\noindent \textit{Step 1.} Let $\delta>0$. Fix $f_0 = f_1 = f^*$ with $f ^*\in \Ff(\gamma,\ell)\cap \Hh_\Dd^\beta(L-\delta)$ and $x\in \Dd$. Then, for large enough $n$, setting $h_n \propto |\TT_n|^{-1/(2\beta+1)}$, we construct a perturbation $(f_{0,n},f_{1,n})$ of $(f_0,f_1)$ defined by
$$
f_{0,n}(y) = f_{1,n}(y) = f^*_n(y) = f^*(y) + a h_n^{\beta} K\big(h_n^{-1}(x-y)\big), \quad y \in \statespace, 
$$
for some smooth kernel $K$ with compact support such that $K(0) =1$, and for some $a=a_{\delta,K}>0$ chosen in such a way that $f_n\in  \Ff(\gamma,\ell)\cap \Hh_\Dd^\beta(L)$. Note that at point $y = x$, $|f_{0,n}(x) - f_0(x) | = |f_{1,n}(x) - f_1(x) | = a_{\delta,K} h_n^{\beta} = a_{\delta,K}  |\TT_n|^{-\beta/(2\beta+1)}$.
\vip
\noindent \textit{Step 2.} 
In the sequel, to shorten expressions, we set
$$
\big| (\widehat{f}_{0,n},\widehat{f}_{1,n}) - f \big|  = \big| \widehat{f}_{0,n}(x) - f(x) \big| +  \big| \widehat{f}_{1,n}(x) - f(x) \big|.
$$
For arbitrary estimators $(\widehat{f}_{0,n}(x),\widehat{f}_{1,n}(x))$ and a constant $C>0$, the maximal risk is bounded below by
\begin{align*}
& \max_{f \in \{ f^*;f^*_n \}} \PP^n_{f} \Big(  |\TT_n|^{\beta/(2\beta+1)} \big| (\widehat{f}_{0,n},\widehat{f}_{1,n}) - f \big|  \geq C\Big) \\
\geq & \frac{1}{2} \Big( \PP^n_{f^*} \big(  |\TT_n|^{\beta/(2\beta+1)}   \big|(\widehat{f}_{0,n},\widehat{f}_{1,n})- f^* \big|  \geq C  \big)  + \PP^n_{f^*_n} \big(  |\TT_n|^{\beta/(2\beta+1)}  \big|(\widehat{f}_{0,n},\widehat{f}_{1,n})- f^*_n \big|  \geq C  \big) \Big) \\
\geq & \frac{1}{2} \EE^n_{f^*} \Big[ {\bf 1}_{\{|\TT_n|^{\beta/(2\beta+1)}  |(\widehat{f}_{0,n},\widehat{f}_{1,n})-f^* | \geq C\}} + {\bf 1}_{\{|\TT_n|^{\beta/(2\beta+1)}  |(\widehat{f}_{0,n},\widehat{f}_{1,n})- f^*_n | \geq C\}} \Big] 
- \frac{1}{2} \| \PP^n_{f^*} - \PP^n_{f^*_n}\|_{TV}.
\end{align*}
By the triangle inequality, we have
\begin{multline*}
|\TT_n|^{\beta/(2\beta+1)}  \Big( \big|(\widehat{f}_{0,n},\widehat{f}_{1,n})-f^* \big| + \big|(\widehat{f}_{0,n},\widehat{f}_{1,n})-f^*_n \big|  \Big) \geq 
2 |\TT_n|^{\beta/(2\beta+1)}  \big| f^*_n(x) - f^*(x) \big| = 2 a_{\delta,K}
\end{multline*}
by Step 1, so if we now take $C<a_{\delta,K}/4$, one of the two indicators within the expectation above must be equal to one with full $\PP_{f^*}^n$-probability. In that case,
$$
\max_{f \in \{ f^*;f^*_n \}} \PP^n_{f} \Big(  |\TT_n|^{\beta/(2\beta+1)} \big| (\widehat{f}_{0,n},\widehat{f}_{1,n}) - f \big|  \geq C\Big)  \geq 
\frac{1}{2} \big( 1 - \| \PP^n_{f^*} - \PP^n_{f^*_n}\|_{TV} \big)
$$
and Theorem~\ref{thm:lower_rateBAR} is thus proved if $\limsup_{n\rightarrow\infty} \| \PP^n_{f^*} - \PP^n_{f^*_n}\|_{TV} < 1$.
\vip
\noindent \textit{Step 3.} By the Pinsker inequality, we have $ \| \PP^n_{f^*} - \PP^n_{f^*_n}\|_{TV} \leq \tfrac{\sqrt{2}}{2} \Big( \EE^n_{f^*} \Big[ \ln \frac{d\PP^n_{f^*}}{d\PP^n_{f^*_n}} \Big] \Big)^{1/2}$ and the log-likelihood ratio can be written
\begin{align*}
\EE^n_{f^*} \Big[ \ln \Big( \frac{d\PP^n_{f^*}}{d\PP^n_{f^*_n}} \Big) \Big] & =  \EE^n_{f^*} \Big[ \sum_{u\in\TT_n} \ln \Big( \frac{\noisedensity\big(X_{u0}-f^*(X_u) , X_{u1}-f^*(X_u) \big)}{\noisedensity\big(X_{u0}-f^*_n(X_u) , X_{u1}-f^*_n(X_u) \big)} \Big) \Big] \\
 & = - \sum_{u\in\TT_n}  \EE^n_{f^*} \Big[\frac{(\sigma_1^2-\sigma_0\sigma_1\rho) \ep_{u0}+ (\sigma_0^2-\sigma_0\sigma_1\rho) \ep_{u1}}{\sigma_0^2\sigma_1^2(1-\rho^2)} (f^*_n-f^*)(X_u) \Big] \\
 & \hspace{4cm} +  \frac{\sigma_0^2+\sigma_1^2-2 \sigma_0\sigma_1 \rho}{2\sigma_0^2\sigma_1^2(1-\rho^2)}  \sum_{u\in\TT_n}  \EE^n_{f^*} \Big[(f^*_n-f^*)^2(X_u) \Big],
\end{align*}
since $\noisedensity$ is chosen to be the bivariate Gaussian density and, under $\PP^n_{f^*}$, we know $X_{u0} = f^*(X_u)+\ep_{u0}$ and $X_{u1} = f^*(X_u)+\ep_{u1}$.
Recall now that $X_{u}$ is independent of $(\ep_{u0},\ep_{u1})$ which is centered. Thus
\begin{align}
\EE^n_{f^*}\Big[ \ln \Big( \frac{d\PP^n_{f^*}}{d\PP^n_{f^*_n}} \Big) \Big] 
& =  \frac{\sigma_0^2+\sigma_1^2-2 \sigma_0\sigma_1 \rho}{2\sigma_0^2\sigma_1^2(1-\rho^2)} \sum_{u\in\TT_n} \EE^n_{f^*}\Big[ (f^*-f^*_n)^2(X_u)\Big] \notag \\
& = \frac{\sigma_0^2+\sigma_1^2-2 \sigma_0\sigma_1 \rho}{2\sigma_0^2\sigma_1^2(1-\rho^2)} \sum_{m=0}^n |\GG_m|  \mu\Big( \Qq_{f^*}^m\big((f^*-f^*_n)^2\big) \Big) \label{eq:NBARloglik}
\end{align}
using the many-to-one formula \eqref{eq:Mto11}, with 
$$
\Qq_{f^*}(x,y) =  \frac{1}{2}\Big(G_0\big(y-f^*(x)) + G_1\big(y-f^*(x)\big) \Big),
$$
where $G_0$ and $G_1$ are the marginals of $\noisedensity$. Since
$$
\Qq_{f^*}\big((f^*-f^*_n)^2\big)(y) = a_{\delta,K}^2 h_n^{2\b} \int_\statespace K^2\big( h_n^{-1} (x-z)\big) \Qq_{f^*}(y,z) dz \leq  a_{\delta,K}^2  |K|_2^2 |\Qq_{f^*}|_{\statespace,\statespace} h_n^{2\beta+1}
$$
where $|\Qq_{f^*}|_{\statespace,\statespace}$ only depends on $\sigma_0, \sigma_1$ and $\rho$. Gathering this last upper-bound together with \eqref{eq:NBARloglik} and the Pinsker inequality, we finally get, with our choice of $h_n$,
 $$ \| \PP^n_{f^*} - \PP^n_{f^*_n}\|_{TV} \lesssim a_{\delta,K}^2$$
and this term can be made arbitrarily small by picking $a_{K,\delta}$ small enough.

\subsection{Proof of Proposition~\ref{prop:asymptotic_as}}
For the choice $h_n \propto |\TT_n|^{-\a}$ with $\a \in (0,1)$, relying successively on  \eqref{termM} and \eqref{termL}, we deduce
$$
\sum_{n\geq 0} \EE_\mu \Big[  \big( M_{\iota ,n}(x) \big)^2 \Big]  < \infty  \quad \text{and} \quad \sum_{n\geq 0} \EE_\mu \Big[ \big( L_{\iota ,n}(x) - (\nu f_\iota) (x) \big)^2 \Big] < \infty.
$$
Thus, as $n\rightarrow\infty$,
$$
M_{\iota ,n}(x) \rightarrow 0 , \quad \PP_\mu - a.s. \quad \text{and} \quad L_{\iota ,n}(x) \rightarrow (\nu f_\iota) (x) , \quad \PP_\mu - a.s. 
$$
From Proposition~\ref{prop:nuhat_ps} and since $\varpi_n\rightarrow 0$,
$$
\widehat{\nu}_n(x) \vee \varpi_n \rightarrow \nu(x) , \quad \PP_\mu - a.s.
$$
We conclude reminding that $$\widehat{f}_{\iota,n}(x) = \frac{M_{\iota ,n}(x)  + L_{\iota ,n}(x) }{\widehat{\nu}_n(x)\vee \varpi_n}.$$

\subsection{Proofs of Theorems~\ref{thm:asymptotic_normality} and~\ref{thm:asymptotic_normality2}}

\begin{proof}[Proof of Theorem~\ref{thm:asymptotic_normality}]
Set $x$ in the interior of $\Dd$. The strategy is to use the following decomposition, which is slightly different from the one used in the proof of Theorem~\ref{thm:upper_rateBAR},
\begin{multline*}
\sqrt{|\TT_n|h_n}\begin{pmatrix} \widehat{f}_{0,n}(x) - f_{0}(x)\\
\widehat{f}_{1,n}(x) - f_{1}(x)\end{pmatrix} =
\tfrac{1}{\widehat{\nu}_n(x)\vee\varpi_n} \bigg\{ \sqrt{|\TT_n|h_n}\begin{pmatrix} M_{0,n}(x) \\ M_{1,n}(x) \end{pmatrix} \\+
 \sqrt{|\TT_n|h_n}\begin{pmatrix} N_{0,n}(x) \\ N_{1,n}(x) \end{pmatrix} +
 \sqrt{|\TT_n|h_n}\begin{pmatrix} R_{0,n}(x) \\ R_{1,n}(x) \end{pmatrix}  \bigg\}
\end{multline*}
where, for $\iota \in \{ 0,1 \}$, $M_{\iota ,n}(x)$ is defined by \eqref{eq:partmart},
\begin{align}
N_{\iota,n}(x) & = \frac{1}{|\TT_n|} \sum_{u \in \TT_n}  K_{h_n}(x-X_u) \big(f_{\iota}(X_u)-f_{\iota}(x)\big), \label{eq:partyneg} \\
R_{\iota,n}(x) & = \Big( \widehat{\nu}_n(x) - \big(\widehat{\nu}_n(x) \vee \varpi_n \big)\Big) f_{\iota}(x), \label{eq:partyrest}
\end{align}
and $\widehat{\nu}_n(x)$ is defined by \eqref{nuhat}. 
The first part of the decomposition is called main term, the second part negligible term and the third part is a remainder term due to the truncation of the denominator of the estimators.
The strategy is the following: prove first that the last two terms goes to zero almost surely and prove a central limit theorem for the main term in a second step.
\vip
\noindent \textit{Step 1. Negligible and remainder terms,  $N_{\iota,n}(x)$ and  $R_{\iota,n}(x)$.} 
We use the decomposition  $N_{\iota,n}(x) = N_{\iota,n}^{(1)}(x)  + N_{\iota,n}^{(2)}(x)$
where
\begin{equation}
N_{\iota,n}^{(1)}(x)= \frac{1}{|\TT_n|} \sum_{u \in \TT_n}  \EE_\nu\big[K_{h_n}(x-X_u) \big(f_{\iota}(X_u)-f_{\iota}(x)\big)\big], \label{eq:NI} 
\end{equation}
\begin{equation}
N_{\iota,n}^{(2)}(x)= \frac{1}{|\TT_n|} \sum_{u \in \TT_n}  \Big( K_{h_n}(x-X_u) \big(f_{\iota}(X_u)-f_{\iota}(x)\big) - \EE_\nu\big[K_{h_n}(x-X_u) \big(f_{\iota}(X_u)-f_{\iota}(x)\big] \Big).
\label{eq:NII} 
\end{equation}
We claim that 
$$ 
\sqrt{|\TT_n|h_n} \,N_{\iota,n}^{(1)}(x) \longrightarrow 0  \quad \text{and} \quad \sqrt{|\TT_n|h_n} \, N_{\iota,n}^{(2)}(x) \rightarrow 0 , \quad \PP_\mu - a.s.
$$
as $n\rightarrow \infty$.
\vip
\noindent \textit{Step 1.1.} Set 
\begin{equation} \label{eq:Hn3}
H_n (\cdot) = K\big(h_n^{-1}(x-\cdot)\big) \big(f_{\iota}(\cdot)-f_{\iota}(x)\big)
\end{equation}
with $h_n$ sufficiently small such that $x-h_n\text{supp}(K) \subset \Dd$.
By the many-to-one formula \eqref{eq:Mto11}, after a decomposition of the subtree $\TT_n$ in $\cup_{m = 0}^n \GG_m$,
\begin{equation} \label{trouverlebiais0}
N_{\iota,n}^{(1)}(x) = \frac{1}{|\TT_n|} \sum_{m=0}^n |\GG_m| \, \EE_\nu\big[K_{h_n}(x-Y_m) \big(f_{\iota}(Y_m)-f_{\iota}(x)\big)\big] = h_n^{-1} \nu(H_n)
\end{equation}
since $\nu$ is the invariant measure of the tagged-branch chain $(Y_m)_{m\geq 0}$ and
\begin{align*}
\nu(H_n) & = \int_\statespace K\big(h_n^{-1}(x-y)\big) \big(f_{\iota}(y)-f_{\iota}(x)\big) \nu(y) dy \\
& = h_n \int_\statespace K(y) \big( f_{\iota}(x-h_ny)-f_{\iota}(x) \big) \nu(x-h_ny) dy \\
& = h_n \int_\statespace K(y) \Big(  \big( (\nu f_{\iota})(x-h_ny)- (\nu f_{\iota})(x) \big) -  \big( \nu(x-h_ny)- \nu(x) \big) f_{\iota}(x) \Big) dy .
\end{align*}
We now use that both $(\nu f_\iota)$ and $f_\iota$ have derivatives up to order $\lfloor \beta \rfloor$. Also remind that $K$ is of order $n_0 > \beta$. By a Taylor expansion, for some $\vartheta$ and $\vartheta' \in (0,1)$,
\begin{align}
\nu(H_n) & = h_n \int_\statespace K(y) \Big( \frac{(-h_n y)^{\lfloor \beta \rfloor}}{\lfloor \beta \rfloor !} (\nu f_{\iota})^{\lfloor \beta \rfloor}(x-\vartheta h_ny)  - \frac{(-h_n y)^{\lfloor \beta \rfloor}}{\lfloor \beta \rfloor !} \nu^{\lfloor \beta \rfloor}(x-\vartheta' h_ny) f_{\iota}(x) \Big)  dy \label{trouverlebiais}\\
& = h_n \int_\statespace K(y) \frac{(-h_n y)^{\lfloor \beta \rfloor}}{\lfloor \beta \rfloor !}\Big( \big( (\nu f_{\iota})^{\lfloor \beta \rfloor}(x-\vartheta h_ny) - (\nu f_{\iota})^{\lfloor \beta \rfloor}(x)\big) \notag \\
& \hspace{6.5cm} -  \big(\nu^{\lfloor \beta \rfloor}(x-\vartheta' h_ny)-\nu^{\lfloor \beta \rfloor}(x)\big) f_{\iota}(x)\Big) dy. \notag
\end{align}
Thus, using $(\nu f_{\iota}) \in \Hh^\b_\Dd(L'')$ and $\nu \in \Hh^\b_\Dd(L'')$ for some $L''> 0$,
$$
|\nu(H_n)|  \leq h_n \int_\statespace |K(y)| \frac{|h_n y|^{\lfloor \beta \rfloor}}{\lfloor \beta \rfloor !} \Big( \big(L''|\vartheta h_n y|^{\{\b\}} \big) + \big(L''|\vartheta' h_n y|^{\{\b\}} \big) f_{\iota}(x) \Big)dy \lesssim h_n^{1+\beta}.
$$
Hence, recalling \eqref{trouverlebiais0}, $N_{\iota,n}^{(1)}(x) \lesssim h_n^{\beta}$ and $\sqrt{|\TT_n|h_n} \,N_{\iota,n}^{(1)}(x)$ goes to zero when $n$ goes to infinity choosing $h_n \propto |\TT_n|^{-\a}$ with $\a>1/(1+2\beta)$.
\vip
\noindent \textit{Step 1.2.} 
In the same way we have proved $|\nu(H_n)| \lesssim h_n^{1+ \b }$, we prove $|\Qq H_n(y)| \lesssim h_n^{1+ \b }$ using the fact that $z\leadsto f_\iota(z)\Qq(y,z)$ and $z\leadsto \Qq(y,z)$ belong to $\Hh^\b_\Dd(L'')$ for some other $L''>0$ for any fixed $y\in\statespace$.
This enables us to reinforce the inequality of Lemma~\ref{lem:gagnerh}(i) and using Lemma~\ref{lem:gagnerh}(ii) we obtain
$$
|\Qq^{l} H_{n}(y) - \nu(H_n)|  \lesssim h_n^{1+ \b } \wedge \big(\VV(y) \rho^{l}\big) \, , \quad l \geq 1.
$$
It brings the following upper-bound using the same technique as in Step 1 and Step 2 of Proposition~\ref{prop:nuhat_L2}:
$$\EE_\mu\big[\big(N_{\iota,n}^{(2)}(x)\big)^2\big] \lesssim h_n^{\b} \big(|\TT_n| h_n \big)^{-1} ,$$ which, by the Borel-Cantelli lemma, leads to the $\PP_\mu$- almost sure convergence of $\sqrt{|\TT_n|h_n} N_{\iota,n}^{(2)}(x)$ to zero, choosing $h_n \propto |\TT_n|^{-\a}$ with $\a > 0$. 
\vip
\noindent \textit{Step 1.3.} To end the first step of the proof, we prove that the remainder term is such that
\begin{equation} \label{eq:rest}
\sqrt{|\TT_n|h_n} R_{\iota,n}(x) \rightarrow 0 , \quad \PP_\mu - a.s.
\end{equation}
as $n\rightarrow \infty$.
Write
\begin{equation} \label{eq:partyrest2}
R_{\iota,n}(x) = \big( \widehat{\nu}_n(x) - \varpi_n \big) f_{\iota}(x) {\bf 1}_{\{ \widehat{\nu}_n(x) < \varpi_n \}}
\end{equation}
where $\big(\widehat{\nu}_n(x) - \varpi_n\big)$ converges $\PP_\mu$-almost surely to $\nu(x)$. We can easily prove that ${\bf 1}_{\{ \widehat{\nu}_n(x) < \varpi_n \}}$ converges $\PP_\mu$-almost surely to $0$, since $\widehat{\nu}_n(x)$  converges $\PP_\mu$-almost surely to $\nu(x)$ and $\varpi_n\rightarrow 0$,
which means that ${\bf 1}_{\{ \widehat{\nu}_n(x) < \varpi_n \}}$ is null  $\PP_\mu$-almost surely beyond some integer. So
$$
\sqrt{|\TT_n| h_n} {\bf 1}_{\{ \widehat{\nu}_n(x) < \varpi_n \}}   = 0 , \quad \PP_\mu - a.s.
$$
beyond some integer and \eqref{eq:rest} is thus proved.
\vip
\noindent \textit{Step 2. Main term $M_{\iota,n}(x)$.} 
We will make use of the central limit theorem for martingale triangular arrays (see for instance Duflo \cite{Duflo98}, Theorem 2.1.9, p 46). 
We follow Delmas and Marsalle \cite{DM10} (section 4) in order to define the notion of the $n$ first individuals of~$\TT$.
Let $( \Pi^*_m)_{m\geq1}$ be independent random variables, where for each~$m$, $\Pi^*_m$ is uniformly distributed over the set of permutations of~$\GG_m$.
The collection $\big( \Pi^*_m(1), \ldots, \Pi^*_m(|\GG_m|) \big)$ is a random drawing without replacement of all the elements of $\GG_m$. 
For $k \geq 1$, set $\rho_k = \inf \{ k' \geq 0 \, , k \leq |\TT_{k'}| \}$ (it can be seen as the number of generation to which belongs the $k$-th element of~$\TT$). We finally define a random order on $\TT$ through $\widetilde{\Pi}$ the function from $\{ 1, 2, \ldots\}$ to $\TT$ such that $\widetilde{\Pi}(1) = \emptyset$ and for $k \geq 2$, $\widetilde{\Pi}(k) = \Pi^*_{\rho_k} (k - |\TT_{\rho_k -1}|)$.
We introduce the filtration $\Gg =\left(\Gg_{n},n \geq 0 \right)$ defined by $\GG_0 =\sigma(X_{\emptyset})$ and for each $n \geq 1$,
\begin{equation*}
\Gg_{n} =\sigma\Big( \big((X_{\widetilde{\Pi}(k)}, X_{(\widetilde{\Pi}(k),0)}, X_{(\widetilde{\Pi}(k),1)}) , 1\leq
k\leq n\big), (\widetilde{\Pi}(k), 1\leq k\leq n+1 ) \Big).
\end{equation*}
For $n\geq 1$, we consider the vector of bivariate random variables $\overline{E}^{(n)}(x) = (\overline{E}^{(n)}_k(x) , 1 \leq k \leq |\TT_n| )$ defined by
\begin{equation} \label{eq:triangarray}
\overline{E}^{(n)}_k(x)= \sum_{l=1}^{k} E^{(n)}_l(x) \quad \text{with} \quad E^{(n)}_l(x) = (|\TT_n|h_n)^{-1/2}
\begin{pmatrix}K\big(h_n^{-1}(x-X_{\widetilde{\Pi}(l)})\big)
\ep_{(\widetilde{\Pi}(l),0)} \\
K\big(h_n^{-1}(x-X_{{\widetilde{\Pi}(l)}})\big) \ep_{(\widetilde{\Pi}(l),1)}\end{pmatrix}.
\end{equation}
Notice that $\overline{E}^{(n)}(x)$ is a square-integrable martingale adapted to $\Gg = (\Gg_n)_{n\geq 0}$. Then, $(\overline{E}^{(n)}(x) ,$ $n \geq 1 )$ is a square-integrable  $\Gg$-martingale triangular array whose
bracket is given by
\begin{align*}
\langle \overline{E}^{(n)}(x)\rangle_{k} =
\sum_{l=1}^{k} \EE \Big[E^{(n)}_l(x) \big(E^{(n)}_l(x)\big)^t \Big|\Gg_{l-1}\Big]  = \Big(\frac{1}{|\TT_n|h_n}\sum_{l=1}^{k} K^{2}\big(h_n^{-1}(x-X_{\widetilde{\Pi}(l)})\big)\Big)\Gamma
\end{align*}
where $\Gamma$ is the noise covariance matrix.
We apply Proposition~\ref{prop:nuhat_ps} (with $|K|_2^{-2} K^2$ replacing $K$ as kernel function) and we obtain
\begin{equation} \label{eq:crochet}
\langle \overline{E}^{(n)}(x)\rangle_{|\TT_n|} =  \Big(\frac{1}{|\TT_n|h_n}\sum_{u \in \TT_n} K^{2}\big(h_n^{-1}(x-X_u)\big)\Big)\Gamma \rightarrow | K |_2^2 \nu(x) \, \Gamma , \quad \PP_\mu - a.s.
\end{equation}
as $n \rightarrow \infty$. Condition (A1) of Theorem 2.1.9 of \cite{Duflo98} is satisfied, this is exactly \eqref{eq:crochet}. Since the bivariate random variables $\big( (\ep_{u0},\ep_{u1}),u  \in\TT \big)$ are independent and identically distributed and since $\ep_0$ and $\ep_1$ have finite moment of order four,
\begin{align*} 
\sum_{k = 1}^ {|\TT_n|} \EE\Big[\|\overline{E}^{(n)}_{k}(x) - \overline{E}^{(n)}_{k-1}(x)\|^{4} \Big|\Gg_{k-1}\Big] & = \frac{\EE\big[ (\ep_0^2 + \ep_1^2)^2\big]}{(|\TT_n|h_n)^{2}} \sum_{k = 1}^
{|\TT_n|} K^{4}\big(h_n^{-1}(x-X_{\widetilde{\Pi}(k)})\big) \\
& =  \frac{\EE\big[ (\ep_0^2 + \ep_1^2)^2\big]}{|\TT_n|h_n} \Big( \frac{1}{|\TT_n|h_n} \sum_{u \in \TT_n} K^{4}\big(h_n^{-1}(x-X_u)\big) \Big)
\end{align*}
where $\|\cdot\|$ denotes the Euclidian norm for vectors and setting $\overline{E}^{(n)}_0(x) = 0$.
We apply Proposition~\ref{prop:nuhat_ps} (with $|K^2|_2^{-2} K^4$ replacing $K$ as kernel function) and we conclude that
\begin{equation} \label{eq:EEM4}
\sum_{k =1}^{\TT_n} \EE\Big[\|\overline{E}^{(n)}_{k}(x) - \overline{E}^{(n)}_{k-1}(x)\|^{4}\Big|\Gg_{k-1}\Big]\rightarrow 0 , \quad \PP_\mu - a.s.
\end{equation}
The Lyapunov condition \eqref{eq:EEM4} implies the Lindeberg condition (A2) of Theorem 2.1.9 of \cite{Duflo98} (see section 2.1.4, p 47 of \cite{Duflo98}).  Therefore, by the central limit theorem for martingale triangular arrays, 
\begin{equation*}
\overline{E}^{(n)}_{|\TT_n|}(x) 
= \sqrt{|\TT_n|h_n}\begin{pmatrix} M_{0,n}(x) \\ M_{1,n}(x) \end{pmatrix}
\toL
\Nn_{2}\big(\boldsymbol{0}_2, | K |_2^2 \nu(x) \, \Gamma\big).
\end{equation*}
\vip
We conclude gathering Step 1 and Step 2, together with $$\widehat{\nu}_n(x) \vee \varpi_n \rightarrow \nu(x) , \quad \PP_\mu - a.s.$$ and the Slutsky lemma.
\vip
\noindent \textit{Step 3. Independence.}  Let $x_1 \neq x_2 \in \Dd$.
We repete Step 2 for
\begin{equation*}
\overline{E}_{k}^{(n)}(x_1,x_2) = \sum_{l=1}^{k} E^{(n)}_l(x_1,x_2) \quad \text{with} \quad E^{(n)}_l(x_1,x_2) = 
\begin{pmatrix}
E^{(n)}_l(x_1) \\ E^{(n)}_l(x_2)
\end{pmatrix}
\end{equation*}
and we are led to
\begin{equation*}
\overline{E}^{(n)}_{|\TT_n|}(x_1,x_2) = 
\sqrt{|\TT_n|h_n}\begin{pmatrix} M_{0,n}(x_1) \\ M_{1,n}(x_1) \\  M_{0,n}(x_2) \\ M_{1,n}(x_2) \end{pmatrix}
\toL
\Nn_{4}\left(\boldsymbol{0}_4, | K |_2^2  \begin{pmatrix}\nu(x_1) \, \Gamma  & \boldsymbol{0}_{2,2} \\ \boldsymbol{0}_{2,2} & \nu(x_2) \, \Gamma \end{pmatrix} \right)
\end{equation*}
with $ \boldsymbol{0}_{4}$ and $\boldsymbol{0}_{2,2}$ respectively the null vector of size 4 and the null matrix of size $2\times 2$,
which shows asymptotic independence between $\big(M_{0,n}(x_1),M_{1,n}(x_1)\big)$ and $\big(M_{0,n}(x_2),M_{1,n}(x_2)\big)$ and thus between $\big(\widehat{f}_{0,n}(x_1),\widehat{f}_{1,n}(x_1)\big)$ and $\big(\widehat{f}_{0,n}(x_2),\widehat{f}_{1,n}(x_2)\big)$ as asserted.
\end{proof}

\begin{proof}[Proof of Theorem~\ref{thm:asymptotic_normality2}]
\noindent \textit{Step 1. Case $\kappa<\infty$.} 
We look carefully at the proof of Theorem~\ref{thm:asymptotic_normality} and see that only Step 1.1 has to be reconsidered.
We prove that 
$$ 
\sqrt{|\TT_n|h_n} \,N_{\iota,n}^{(1)}(x) \longrightarrow \kappa \nu(x) \boldsymbol{m}(x).
$$
Indeed for $\beta$ an integer, using \eqref{trouverlebiais0} and \eqref{trouverlebiais}, 
$$
\sqrt{|\TT_n|h_n} \, N_{\iota,n}^{(1)}(x) = \frac{(-1)^{\beta} h_n^{\beta} \sqrt{|\TT_n|h_n} }{ \beta  !} \int_\statespace y^{\beta} K(y)  \big( (\nu f_{\iota})^{ \beta}(x - \vartheta h_ny)  -  \nu^{ \beta }(x - \vartheta' h_ny) f_{\iota}(x) \big)  dy 
$$
and we conclude letting $n$ go to infinity since $(\nu f_{\iota})^{ \beta}$ and $ \nu^{ \beta }$ are continuous.

\vip
\noindent \textit{Step 2. Case $\kappa=\infty$.} With the same argument we prove that in that case,
$$ 
h_n^{-\beta} \,N_{\iota,n}^{(1)}(x) \longrightarrow \nu(x) \boldsymbol{m}(x)
$$
as $n\rightarrow \infty$. 
Looking at Steps 1.2 and 1.3 in the proof of Theorem~\ref{thm:asymptotic_normality}, we obtain
$$ 
h_n^{-\beta} \,N_{\iota,n}^{(2)}(x) \longrightarrow 0 \quad \text{and} \quad h_n^{-\beta} \,R_{\iota,n}(x) \longrightarrow 0 , \quad \PP_\mu - a.s.
$$
In addition,
$$
h_n^{-\beta} \,M_{\iota,n}(x) = \big( h_n^{\beta} \sqrt{|\TT_n|h_n} \big)^{-1}   \, \sqrt{|\TT_n|h_n} \,M_{\iota,n}(x)  \overset{\PP_\mu}{\longrightarrow} 0,
$$
since we work in the case  $h_n^\b \sqrt{|\TT_n|h_n} \rightarrow \kappa=\infty$ and  $\sqrt{|\TT_n|h_n} \,M_{\iota,n}(x)$ is asymptotically Gaussian as proved previously (Step 3 in the proof of Theorem~\ref{thm:asymptotic_normality}).
\end{proof}

\begin{rk}
If $\beta$ is not an integer, we could generalize Theorem~\ref{thm:asymptotic_normality2} but at the cost of introducing fractional derivatives. Note that the definition of this notion is not unique (see \cite{Trujillo} or \cite{Usero}). We restrict the parameter $\beta$ to be an integer in order to avoid here additional technicalities.
\end{rk}

\subsection{Proofs of Corollary \ref{cor:TCLBierens} and Theorem~\ref{thm:test}}
\begin{proof}[Proof of Corollary \ref{cor:TCLBierens}]
On the one hand, applying Theorem~\ref{thm:asymptotic_normality2} to the estimator $\widehat{f}_{\iota,n}^{(a)}$ built with the bandwidth $h_n^{(a)} \propto |\TT_n|^{-1/(2\beta+1)}$,
\begin{equation*}
|\TT_n|^{\frac{\b}{2\b+1}} A_n(x) 
\toL
 \Nn_{2}\big(  \boldsymbol{m}_2(x) \, , \boldsymbol{\Sigma}_2(x)\big)
\quad \text{with} \quad  A_n(x)  = \begin{pmatrix} \widehat{f}_{0,n}^{(a)}(x)-f_{0}(x)\\ \widehat{f}_{1,n}^{(a)}(x)-f_{1}(x)\end{pmatrix}.
\end{equation*}
On the other hand, applying Theorem~\ref{thm:asymptotic_normality2} to the estimator $\widehat{f}_{\iota,n}^{(b)}$ built with the bandwidth $h_n^{(b)} \propto |\TT_n|^{-\delta/(2\beta+1)}$ for $\delta\in (0,1)$,
$$
|\TT_n|^{ \frac{\delta \b}{2\beta+1}} B_n(x) \overset{\PP_\mu}{\longrightarrow} \boldsymbol{m}_2(x) \quad \text{with} \quad   B_n(x) = \begin{pmatrix} \widehat{f}_{0,n}^{(b)}(x)-f_{0}(x)\\ \widehat{f}_{1,n}^{(b)}(x)-f_{1}(x)\end{pmatrix}.
$$
Combining these two results, we obtain
\begin{multline*}
|\TT_n|^{\frac{\b}{2\b+1}}\begin{pmatrix} \bar{f}_{0,n}(x)-f_{0}(x)\\ \bar{f}_{1,n}(x)-f_{1}(x)\end{pmatrix} =
\frac{ |\TT_n|^{\frac{\b}{2\b+1}}A_n(x) -  |\TT_n|^{ \frac{\delta \b}{2\beta+1}}  B_n(x)}{1-|\TT_n|^{- \frac{(1-\delta)\b}{2\beta+1}}}  \toL \Nn_{2}\big(  \boldsymbol{0}_2 \, , \boldsymbol{\Sigma}_2(x)\big),
\end{multline*}
as announced.
\end{proof}

\begin{proof}[Proof of Theorem~\ref{thm:test}]
The study of the test statistics $W_n(x_1,\ldots,x_k)$ under $\Hh_0$ and $\Hh_1$ then follows classical lines. We give here the main argument for $k= 2$. By \ref{cor:TCLBierens} and using in addition the asymptotical independence stated in Theorem~\ref{thm:asymptotic_normality}, we obtain
$$
|\TT_n|^{\frac{\b}{2\b+1}}\begin{pmatrix} \bar{f}_{0,n}(x_1)-f_{0}(x_1)\\ \bar{f}_{1,n}(x_1)-f_{1}(x_1) \\  \bar{f}_{0,n}(x_2)-f_{0}(x_2)\\ \bar{f}_{1,n}(x_2)-f_{1}(x_2)\end{pmatrix} 
\toL \Nn_4 \big( \boldsymbol{0}_{4} , \boldsymbol{\Sigma}_4(x_1,x_2) \big) \;\; \text{with} \;\;  \boldsymbol{\Sigma}_4(x_1,x_2) = \left(
\begin{array}{cc}
\boldsymbol{\Sigma}_2(x_1)  &   \boldsymbol{0}_{2,2}   \\
\boldsymbol{0}_{2,2}  &   \boldsymbol{\Sigma}_2(x_2)
\end{array}
\right).
$$
Then, using the Delta-method,
$$
|\TT_n|^{\frac{\b}{2\b+1}}\begin{pmatrix} \big( \bar{f}_{0,n}(x_1) -  \bar{f}_{1,n}(x_1)\big) - \big( f_{0}(x_1) -f_{1}(x_1)\big) \\  \big(\bar{f}_{0,n}(x_2)- \bar{f}_{1,n}(x_2)\big)-\big(f_0(x_2) - f_{1}(x_2)\big)\end{pmatrix} 
\toL \Nn_2 \big( \boldsymbol{0}_{2} ,  \boldsymbol{\Sigma}_2(x_1,x_2) \big)
$$
with 
$$ \boldsymbol{\Sigma}_2(x_1,x_2)  = 
|K|_2^2 (\sigma_0^2+\sigma_1^2 - 2\sigma_0 \sigma_1 \varrho)
\left(
\begin{array}{cc}
\big(\nu(x_1)\big)^{-1}  &   0   \\
0  &  \big(\nu(x_2)\big)^{-1}
\end{array}
\right),
$$
which boils down to, under $\Hh_0$,
$$
\frac{|\TT_n|^{\frac{\b}{2\b+1}}}{\sqrt{|K|_2^2 (\sigma_0^2+\sigma_1^2 - 2\sigma_0 \sigma_1 \varrho)}}\begin{pmatrix} \nu(x_1)^{1/2} \big( \bar{f}_{0,n}(x_1) -  \bar{f}_{1,n}(x_1)\big) \\  \nu(x_2)^{1/2} \big(\bar{f}_{0,n}(x_2)- \bar{f}_{1,n}(x_2)\big) \end{pmatrix} 
\toL \Nn_2 \big( \boldsymbol{0}_2 ,  \boldsymbol{I}_2 \big),
$$
with $\boldsymbol{I}_2$ the identity matrix of size $2\times 2$.
The replacement of $\nu(\cdot)$ by its estimator $\widehat{\nu}_n(\cdot)$ is licit by the Stlutsky theorem.
Thus, under $\Hh_0$, $W_n(x_1,x_2)\toL \chi^2(2)$, the chi-squared distribution with~$2$ degrees of freedom. Under $\Hh_1$, we prove that $W_n(x_1,x_2)$ converges $\PP_\mu$-almost surely to $+ \infty$ following the same lines and using 
$$
\sum_{l=1}^2 \widehat{\nu}_n(x_l) \big( \bar{f}_{0,n}(x_l) -  \bar{f}_{1,n}(x_l)\big)^2  \longrightarrow \sum_{l=1}^2 \nu(x_l)  \big(f_0(x_l) - f_1(x_l)\big)^2  \neq 0 , \quad \PP_\mu - a.s.
$$
when $\Hh_1$ is valid.
\end{proof}

\section{\textsc{Appendix}} \label{sec:appendix}
\subsection{Proof of Lemma~\ref{lem:Mto1}}
Let us first prove (i), see Delmas and Marsalle \cite{DM10} (Lemma 2.1). For $u = (u_1,u_2, \ldots , u_m) \in \GG_m$,
$$
\EE_\mu \big[ g(X_u) \big] = \mu\big(\Pp_{u_1} \Pp_{u2} \ldots \Pp_{u_m}(g)\big).
$$
Then
\begin{align*}
\EE_\mu \big[ \sum_{u\in\GG_m} g(X_u) \big] = \sum_{(u_1,\ldots,u_m)\atop\in \{0,1\}^m} \mu\big(\Pp_{u_1} \Pp_{u2} \ldots \Pp_{u_m}(g)\big) & = \mu\Big( \sum_{(u_1,\ldots,u_m)\atop\in \{0,1\}^m} \Pp_{u_1} \Pp_{u2} \ldots \Pp_{u_m}(g)\Big) \\
& = \mu \big( 2^m (\Pp_0 + \Pp_1 )^m (g) \big) = |\GG_m| \mu(\Qq^m g)
\end{align*}
since $\Qq = (\Pp_0+\Pp_1)/2$ and $|\GG_m| = 2^m$. We also know that $\mu(\Qq^m g) = \EE_\mu[g(Y_m)]$.\\

We now turn to (ii). We refer to Guyon \cite{Guyon} for another strategy of proof (proof of Equation (7)).
Some notation first: for $u = (u_1,\ldots,u_m)$ and $v=(v_1,\ldots,v_n)$ in $\TT$, we write $uv = (u_1,\ldots,u_m,v_1,\ldots,v_n)$ for the concatenation. For $m\geq 0$, we denote by $\Ff_m$ the sigma-field generated by $(X_u, |u|\leq m)$.

For $m \geq 1$, whenever $u\neq v\in \GG_m$, there exist $w\in \TT_{m-1}$ together with $i\neq j \in \{0,1\}$ and $\tilde{u}, \tilde{v}$ such that $u = wi\tilde{u}$ and $v = wj\tilde{v}$, where we call $w$ the most recent common ancestor of $u$ and $v$.
The main argument uses consecutively a first conditioning by $\Ff_{|w|+1}$ which lets $X_u$ and $X_v$ conditionally independent and a conditional many-to-one formula of kind (i), a second conditioning by $\Ff_{|w|}$ and the definition of the $\TT$-transition $\Ttransition$ and finally the many-to-one formula (i),
\begin{align*}
\EE_\mu \Big[ \sum_{u \neq v \in \GG_m} g(X_u) g(X_v) \Big]  & =  2 \sum_{l=1}^m \sum_{w\in \GG_{m-l}} \EE_\mu \big[ \sum_{\tilde{u} \in \GG_{l-1} \atop \tilde{v} \in \GG_{l-1}} g(X_{w0\tilde{u}})g(X_{w1\tilde{v}})\big]\\
& = 2 \sum_{l=1}^m \sum_{w\in \GG_{m-l}} |\GG_{l-1}|^2 \EE_\mu  \Big[ \Qq^{l- 1} g(X_{w0}) \Qq^{l - 1} g(X_{w1}) \Big] \\
& =  2 \sum_{l=1}^m (2^{l-1})^2\sum_{w\in \GG_{m-l}}  \EE_\mu \Big[ \Ttransition ( \Qq^{l - 1} g \otimes \Qq^{l - 1} g ) ( X_{w} ) \Big] \\
& = 2 \sum_{l = 1}^m (2^{l -1})^2 |\GG_{m-l}| \EE_\mu \Big[   \Ttransition ( \Qq^{l  - 1} g \otimes \Qq^{l  - 1} g ) ( Y_{m-l} ) \Big]  \\
& =  2^m \sum_{l = 1}^m 2^{l-1} \mu \Big( \Qq^{m-l} \big( \Ttransition ( \Qq^{l - 1} g \otimes \Qq^{l  - 1} g ) \big) \Big),
\end{align*}
as asserted.

\subsection{Proof of Lemma~\ref{lem:gagnerh}}
First,
\begin{align*}
| \Qq H_n (y) | & \leq \int_\statespace | F\big(h_n^{-1}(x-z)\big)| |G(z)| \Qq(y,z)dz  \\
& = h_n \int_{{\rm supp}(F)} | F(z)| |G(x-h_n z)| \Qq(y,x-h_n z)dz  \leq h_n \, |G|_\Dd |\Qq|_{\statespace,\Dd} |F|_1
\end{align*}
for $h_n$ such that $x-h_n {\rm supp}(F) \subset \Dd$ (remind that $x$ belongs to the interior of $\Dd$). We prove in the same way the bound on $\nu (H_n)$.
Hence we have proved (i) and we now turn to (ii). The first bound $h_n$ obviously comes from (i) and it remains to prove the second bound $\VV(y) \rho^m$. Under Assumption~\ref{ass:ergodicity2}, we apply Proposition~\ref{lem:ergodicity} to $g = H_n/|H_n|_\infty$ and it brings 
$$
|\Qq^m H_n(y) - \nu(H_n)| \leq R |H_n|_\infty \rho^m \VV(y).
$$
Since $|H_n|_\infty = |F|_\infty |G|_\Dd$ for $h_n$ such that $x-h_n {\rm supp}(F) \subset \Dd$, we obtain the announced upper-bound.
%

\subsection{Proof of Lemma~\ref{lem:minorationnu}} For every $|z|\leq d$, 
$$
\nu(z) = \int_{\statespace} \nu(y) \Qq(y,z) dy \geq \inf_{|y|\leq M_2 , \atop |z|\leq d} \Qq(y,z) \int_{|y|\leq M_2} \nu(y) dy.
$$
On the one hand,
$$
\inf_{|x|\leq M_2 , \atop |y|\leq d} \Qq(x,y) \geq \frac{1}{2} \Big\{ \inf_{|x|\leq M_2 , \atop |y|\leq d} G_0\big(y-f_0(x)\big) +  \inf_{|x|\leq M_2 , \atop |y|\leq d} G_1\big(y-f_1(x)\big)\Big\} \geq \delta\big(d + (\ell + \gamma M_2)\big) \geq \delta(M_3) >0
$$
if $d>0$ is such that $d +  (\ell + \gamma M_2) \leq M_3$, which is possible by Assumption~\ref{ass:numinor}.
On the other hand,
\begin{align*}
\int_{|y| > M_2} \nu(y) dy & = \int_{|y|>M_2} \int_{x \in \statespace} \nu(x) \Qq(x,y)  dx dy \\
& \leq |\nu|_{\infty}  \int_{|y|>M_2} \int_{x\in\statespace} \frac{1}{2} \Big( G_0\big( y-f_0(x) \big) + G_1\big(y -f_1(x)\big)\Big) dx dy \leq \eta(M_2)
\end{align*}
since $(f_0,f_1)$ belongs to $\Ff(\gamma,\ell)^2$ and $(G_0,G_1)$ to $\Gg(r,\lambda)^2$. We know in addition that $|\nu|_\infty \leq |\Qq|_{\statespace,\statespace} \leq (|G_0|_\infty+|G_1|_\infty)/2<\infty.$ Using $\eta(M_2)<1$ brings
$$
\int_{|y|\leq M_2} \nu(y) dy > 0.
$$
Hence $\nu(y) \geq \delta(M_3) \big( 1 - \eta(M_2) \big) > 0$ for any $|y|\leq d$. We have uniformity in $f_0$ and $f_1$ since $\delta(M_3)$ and $\eta(M_2)$ are uniform bounds on the class $\Ff(\gamma,\ell)^2$.


\vip
\noindent {\bf Acknowledgements.} 
We are grateful to A. Guillin and M. Hoffmann for helpful comments and G. Fort for useful discussions. We thank E. L\"ocherbach and P. Reynaud-Bouret for a careful reading.
We also deeply thank N. Krell for discussions on real data and E. Stewart for sharing his data.
S.V. B.P. thanks the Hadamard Mathematics Labex of the {\it Fondation Math\'ematique Jacques Hadamard} for financial support.



\end{document}